\documentclass[11pt]{amsart}
\usepackage{amsmath,amssymb,latexsym,soul,cite,mathrsfs,amsfonts}
\usepackage{float}
\usepackage{xcolor}
\usepackage{color,enumitem,graphicx}
\usepackage{mathtools}
\usepackage[colorlinks=true,urlcolor=blue,
citecolor=blue,linkcolor=blue,linktocpage,pdfpagelabels,
bookmarksnumbered,bookmarksopen]{hyperref}
\usepackage[english]{babel}
\usepackage{outline}
\usepackage{comment}
\usepackage{frame}
\usepackage[left=1.9cm,right=1.9cm,top=2.0cm,bottom=2.0cm]{geometry}
\usepackage[hyperpageref]{backref}
\usepackage[colorinlistoftodos]{todonotes}
\makeatletter
\providecommand\@dotsep{5}
\def\listtodoname{List of Todos}
\def\listoftodos{\@starttoc{tdo}\listtodoname}
\makeatother
\usepackage{verbatim}
\numberwithin{equation}{section}
\newtheorem{thm}{Theorem}[section]
\newtheorem{prop}[thm]{Proposition}
\newtheorem{lem}[thm]{Lemma}

\newcommand{\R}{\mathbb{R}}
\newcommand{\G}{\mathbb{G}}
\newcommand{\2}{2^*_{Q}}

\newcommand{\h}{S^1_{0}(\Omega)}

\title []{Critical Ambrosetti-Prodi type problems on Carnot groups}
\author[Suman Kanungo]{Suman Kanungo}
\address[]{\newline\indent
	Department of Mathematics
	\newline\indent
	Indian Institute of Technology Bhilai
	\newline\indent
	491002, Durg, Chhattisgarh, India}
\email{\href{sumankau@iitbhilai.ac.in}{sumankau@iitbhilai.ac.in}}

\author[Pawan Kumar Mishra]{Pawan Kumar Mishra}
\address[]{\newline\indent
	Department of Mathematics
	\newline\indent
	Indian Institute of Technology Bhilai
	\newline\indent
	491002, Durg, Chhattisgarh, India}
\email{\href{pawan@iitbhilai.ac.in}{pawan@iitbhilai.ac.in}}
\subjclass[2020]{}
\begin{document}
	\begin{abstract}
		In this paper, we investigate a class of critical Ambrosetti-Prodi type problems involving the sub-Laplacian on a Carnot group. Specifically, we consider
		\[
		\left\{
		\begin{aligned}
			-\Delta_{\mathbb{G}} u &= \lambda u + u_{+}^{2_{Q}^{*}-1} + f(\xi) \quad &&\text{in } \Omega,\\[2mm]
			u &= 0 \quad &&\text{on } \partial\Omega,
		\end{aligned}
		\right.
		\]
		where $\Delta_{\mathbb{G}}$ is the sub-Laplacian on a Carnot group $\mathbb{G}$, $\Omega \subset \mathbb{G}$ is an open bounded domain with smooth boundary, $\lambda>0$ is a real parameter, $f\in L^{\infty}(\Omega)$, $u_{+}$ denotes the positive part of $u$, and $2_{Q}^{*}$ is the critical Sobolev exponent associated with the homogeneous dimension $Q$. Motivated by the classical Ambrosetti-Prodi problem, we establish existence and multiplicity results for the cases $\lambda<\lambda_{1}$ and $\lambda>\lambda_{1}$, where $\lambda_{k}$ denotes the $k$-th Dirichlet eigenvalue of $-\Delta_{\mathbb{G}}$. We also prove the existence of solutions at resonance when $\lambda=\lambda_{1}$ and show that bifurcation occurs from each eigenvalue $\lambda_{k}, k >1$.
	\end{abstract}
	\maketitle
	\section{Introduction}
	Let $\G$ be a Carnot group and let $\Omega \subset \G$ be an open bounded domain with smooth boundary $\partial \Omega$. Let $\2 \coloneq \frac{2Q}{Q-2}$ be the critical Sobolev exponent associated with the Sobolev inequality in $\G$, where $Q$ denotes the homogeneous dimension. For $\lambda>0$, we study the following critical problem
	\begin{equation}
		\label{1.1}
		\left\{ \begin{aligned} 
			&-\Delta_{\G} u = \lambda u + u^{\2-1}_+ + f(\xi) \quad &&\text{ in } \Omega,\\
			& u = 0 \quad &&\text{ on } \partial\Omega,
		\end{aligned} \right.
	\end{equation}
	where $f\in L^\infty(\Omega)$, $u_+ \coloneq \max\{ u, 0\}$ and $\Delta_{\G}$ denotes the sub-Laplacian on $\G$. Our study is motivated by the classical problem introduced by Ambrosetti and Prodi \cite{Ambrosetti},  which describes how the solvability and multiplicity of solutions depend on the interaction between the nonlinearity and the spectrum of the underlying operator. Precisely, they studied the boundary value problem
	\begin{equation}
		\label{1.2}
		\begin{cases}
			-\Delta u = g(u) + f(x), & \text{ in } \Omega, \\
			u = 0, &  \text{ on }\partial \Omega,
		\end{cases}
	\end{equation}
	where $f \in C^{0,\alpha}(\bar{\Omega})$, for some $\alpha \in (0,1)$, and $g \in C^2(\mathbb{R})$ satisfies $g(0)=0$, $g''(s) > 0$ for all $s\in \R$. Defining
	$g^- = \lim_{s\to -\infty} g(s)/s$, $g^+ = \lim_{s\to +\infty} g(s)/s$ and assuming $0<g^-<\mu_1<g^+<\mu_2$, where  $\mu_k$ be the $k$-th eigenvalue of $-\Delta$ with with Dirichlet boundary conditions, they proved the existence of a closed, connected $C^1$-manifold $M\subset C^{0,\alpha}(\bar{\Omega})$ of codimension one such that $C^{0,\alpha}(\bar \Omega)\setminus M$ consists of two components. For $f$ lies in one component the problem admits no solution, in the other exactly two solutions, while for $f\in M$ it admits a unique solution. This result emphasizes the role of the interaction between the asymptotic behavior of the nonlinearity and the spectrum of $-\Delta$. Motivated by this perspective, several authors have studied Ambrosetti-Prodi type problems with unilateral superlinear growth, in particular $g^{-}\in(0,\mu_1)$ and $g^{+}=+\infty$. A prototype example is
	\begin{equation}
		\label{1.3}
		\left\{ \begin{aligned} 
			&-\Delta u = \lambda u + u_+^p + f(x) \quad &&\text{ in } \Omega,\\
			& u = 0 \quad &&\text{ on } \partial\Omega,
		\end{aligned} \right.
	\end{equation}
	where $1 < p < \frac{N+2}{N-2}$, if $N \geq 3$, $1 < p < \infty$ if $N=2$, and $u_+= \max\{u,0 \}$. For $ \lambda <\mu_1$, the nonlinearity $g(s) = (s_+)^p + \lambda s$ crosses the first eigenvalue, yielding a superlinear Ambrosetti-Prodi problem. This case was studied by de Figueiredo \cite{Figueiredo1}, assuming $f\in C^\alpha(\bar \Omega)$ of the form $f=t\phi_1+h$, where $\phi_1$ is the positive normalized eigenfunction associated to $\mu_1$, $h\in C^\alpha(\bar \Omega)$ satisfies $\int_{\Omega} h\phi_1\,dx=0$, and $t\in\mathbb{R}$. When $\lambda>\mu_1$, the nonlinearity does not cross the first eigenvalue, this case was investigated by Ruf and Srikanth \cite{Ruf}. Assuming $\lambda\neq\mu_k$ for all $k\in\mathbb{N}$ and the same structure of the term $f$, they proved the existence of a constant $T(h)$ such that, for $t>T(h)$, problem \eqref{1.3} admits at least two solutions. Their approach relies on the generalized Mountain Pass Theorem or linking theorem. Subsequently, Deng \cite{Deng} considered a superlinear elliptic problem involving the critical growth $u^{\2-1}_+$, together with a perturbation $g(x, u_+)$ of subcritical growth, and proved the existence of multiple solutions. Later, de Figueiredo and Jianfu \cite{Figueiredo} investigated a critical Ambrosetti-Prodi type problem of the form
	\begin{equation*}
		\left\{
		\begin{aligned}
			-\Delta u &= \lambda u + u_{+}^{2^{*}-1} + f(x) \quad &&\text{in } \Omega,\\
			u &= 0 \quad &&\text{on } \partial\Omega,
		\end{aligned}
		\right.
	\end{equation*}
	where $f = t\varphi_{1} + h$ as before, but with $h \in L^{r}(\Omega)$ for some $r>N$. 
	They showed that the problem admits a negative solution for large values of $|t|$, with the sign of $t$ depending on whether $\lambda$ lies below or above the first eigenvalue $\mu_{1}$. When $\lambda$ is not an eigenvalue and the dimension is sufficiently large, the existence of a second solution was established via the linking theorem. In the resonant case $\lambda=\mu_{1}$, they proved that a solution exists provided $f$ is sufficiently small in $L^{2}(\Omega)$. Moreover, they established bifurcation results from any eigenvalue. Ambrosetti-Prodi type problems involving different operators have also been studied in \cite{Ambrosio, Calanchi, Cuesta, Miyagaki, Paiva, Ribeiro, Sharma, Wei} and the references therein.
	
	Despite these developments in the Euclidean setting, to the best of our knowledge Ambrosetti-Prodi type problems, both in the subcritical and critical cases, have not been previously studied in the framework of Carnot groups. In particular, no results appear to be available for problem \eqref{1.1}, which involves the sub-Laplacian and unilateral critical growth. The purpose of the present work is to fill this gap. Before stating our main results, we briefly review some related developments on semilinear equations in Carnot groups.
	
	\subsection{Sub-Laplacians and Semilinear Problems on Carnot Groups}
	Let $\G$ be a Carnot group and $\Delta_{\G}$ denote a sub-Laplacian on $\G$ (see Section~\ref{sec1} for precise definitions). The operator $\Delta_{\G}$ is a second-order differential operator which is degenerate elliptic. The lack of uniform ellipticity leads to substantial analytical difficulties and prevents the direct use of many classical tools from elliptic theory. A crucial feature distinguishing sub-Laplacians on Carnot groups from general degenerate elliptic operators is that $\Delta_{\G}$ can be written as a sum of squares of left-invariant vector fields satisfying H\"ormander’s hypoellipticity condition \cite{Hormander}. This structure yields some regularity properties and plays a fundamental role in the analysis of partial differential equations on Carnot groups. In this direction, Folland \cite{Folland} developed a functional-analytic framework for sub-Laplacians and established the existence of a global fundamental solution. Shortly thereafter, Rothschild and Stein \cite{Rothschild} clarified the role of sums of squares of vector fields in the study of second-order partial differential equations, initiating a systematic study of PDEs on Carnot groups. In the Heisenberg group $\mathbb{H}^n$, which is the simplest nontrivial Carnot group, the associated critical problem
	has been extensively studied. In a series of seminal works, Jerison and Lee \cite{Jerison1,Jerison2,Jerison3} obtained sharp Sobolev inequalities and explicit extremal functions, extending the classical results of Aubin and Talenti. The importance of these results for critical nonlinear equations was emphasized by Br\'ezis and Nirenberg \cite{Brezis} and the corresponding Br\'ezis-Nirenberg type result on $\mathbb{H}^n$ was proved in \cite{Citti}. Further related results on the Heisenberg group are available in \cite{Birindelli, Birindelli1, Brandolini, Citti1, Garagnani, Garofalo1, Lanconelli, Lanconelli1, Lu, Maalaoui, Maalaoui1, Palatucci, Uguzzoni, Uguzzoni1}. For general Carnot groups, explicit minimizers for the critical Sobolev inequality are not known. However, the best Sobolev constant is attained \cite{Garofalo}, and the asymptotic behavior of minimizing sequences has been described in detail, see \cite{Bonfiglioli}. Using these estimates, Loiudice \cite{Loiudice2} established a Br\'ezis-Nirenberg type result for critical equations on Carnot groups.  Further related results on critical problems in Carnot groups can be found in \cite{Bisci, BonfiglioliUguzzoni, Garofalo, Garofalo1}. After reviewing the relevant works on the Carnot group, we are now ready to discuss the assumptions and main results of this paper. 
	
	\subsection{Assumptions and main results}
	
	Throughout this section, we adopt the notation and assumptions introduced in Section~\ref{sec1}. Let $S^1_0(\Omega)$ denotes the Folland-Stein space associated with $\Delta_{\G}$, endowed with the norm and inner product
	\[
	\|u\|_{\h} \coloneq \left( \int_{\Omega} |\nabla_{\G} u|^2 \, d\xi \right)^{1/2}, \quad \langle u, v \rangle_{\h} \coloneq \int_{\Omega} \langle \nabla_{\G} u, \nabla_{\G} v \rangle \, d\xi,
	\]
	for all $u,v \in S^1_0(\Omega)$, where $\nabla_{\G}$ denotes the $\G$-gradient. We say that $u \in S^1_0(\Omega)$ is a weak solution of
	problem \eqref{1.1} if
	\[
	\int_{\Omega} \langle \nabla_{\G} u, \nabla_{\G} \varphi \rangle \, d\xi
	= \lambda \int_{\Omega} u \varphi \, d\xi
	+ \int_{\Omega} u_+^{2^*_Q-1} \varphi \, d\xi
	+ \int_{\Omega} f \varphi \, d\xi
	\quad \text{for all } \varphi \in S^1_0(\Omega).
	\]
	Let $0 < \lambda_1 < \lambda_2 \le \cdots \le \lambda_k \le \lambda_{k+1} \le \cdots$ denote the eigenvalues of the operator $(-\Delta_{\G}, S^1_0(\Omega))$, repeated
	according to their multiplicity. Let $e_1$ be the eigenfunction associated with
	the first eigenvalue $\lambda_1$, normalized by
	$\|e_1\|_{L^2(\Omega)} = 1.$
	As recalled in Section~\ref{sec1}, $e_1$ is positive in $\Omega$. Throughout the paper, we assume that the function $f$ admits the decomposition $f = t e_1 + h,$
	where $t \in \mathbb{R}$ and $h \in L^\infty(\Omega)$ satisfies the orthogonality
	condition
	$\int_{\Omega} h e_1 \, d\xi = 0.$
	
	We are now in a position to state the main results of the paper. The first theorem concerns the existence and multiplicity of solutions to problem \eqref{1.1}, depending on the position of the parameter $\lambda$ with respect to the first eigenvalue of the sub-Laplacian.
	\begin{thm}\label{thm1.1}
		The following assertions hold.\\
		\textnormal{(i)} If $0 < \lambda < \lambda_1$, then there exists
		$t_0 = t_0(h) < 0$ such that, for every $t < t_0$, problem \eqref{1.1}
		admits a nonpositive solution $u_t \in \h$.\\
		\textnormal{(ii)} Let $\lambda > \lambda_1$. If $\lambda$ is not an eigenvalue
		of $-\Delta_\G$, then there exists $t_0 = t_0(h) > 0$ such that, for every
		$t > t_0$, problem \eqref{1.1} admits a nonpositive solution.
		If $\lambda$ is an eigenvalue and
		$h \in \ker(-\Delta_\G - \lambda I)^\perp$, the same conclusion holds.\\
		\textnormal{(iii)} Assume in addition that $Q > 6$ and that $\lambda$ is not an
		eigenvalue of $-\Delta_\G$. Then problem \eqref{1.1} admits a second nontrivial solution.
	\end{thm}
	We next consider the resonant case at the first eigenvalue.
	\begin{thm}\label{thm1.2}
		Assume that $\lambda = \lambda_1$ in problem \eqref{1.1}. 
		Let $f \in L^2(\Omega)$ satisfy	
			$\int_\Omega f e_1 \, d\xi < 0.$
		There exist positive constants $K_1$ and $K_2$ given in equations~\eqref{3.15} and \eqref{3.16}, depending only on 
		$Q$, $\lambda_1$, $\lambda_2$, and the optimal Sobolev constant $S_\G$, such that
		if
		\begin{equation*}
			\|f\|_{L^2(\Omega)} \le K_1
			\quad \text{and} \quad
			-\int_\Omega f e_1 \, d\xi < K_2 ,
		\end{equation*}
		then problem \eqref{1.1} admits at least one weak solution in $\h$.
	\end{thm}
	We now analyze the bifurcation behavior of the branch of nonpositive solutions near higher eigenvalues. The following result shows that $\lambda_k$ with $k>1$, is a bifurcation point for problem \eqref{1.1}.
	\begin{thm} \label{thm1.3}
		Let $h \in \ker(-\Delta_\G - \lambda_kI)^\perp$ with $k > 1$.  
		In the space $\mathbb{R} \times {S}^1_0(\Omega)$, consider the curve $(\lambda, u_t(\lambda))$ 
		of nonnpositive solutions to \eqref{1.1}, obtained from Theorem \ref{thm1.1} for $\lambda$ close to $\lambda_k$. Then $(\lambda_k, u_t(\lambda_k))$ is a bifurcation point.
	\end{thm}
	
	\subsection{Our approach}
	The proofs of Theorem \ref{thm1.1}-\ref{thm1.3} are based on variational methods applied to the energy functional associated with problem \eqref{1.1}, defined on $\h$. We first establish the existence of a nonpositive solution $u_t$ of problem \eqref{1.1}. We observe that for any nonpositive solution the nonlinear term vanishes, and the equation reduces to a linear problem. The existence of such a solution follows from the Fredholm alternative and uniform bounds, for sufficiently large values of the parameter $t$, with the sign depending on the position of $\lambda$ relative to the first eigenvalue. To obtain a second solution, we consider perturbations around the first solution $u_t$. This leads to a new variational formulation whose critical points correspond to weak solutions of problem \eqref{1.1}. For $0<\lambda<\lambda_1$, the associated energy functional exhibits a mountain-pass geometry. For $\lambda>\lambda_1$, with $\lambda$ not an eigenvalue, we use a linking argument based on the spectral decomposition of the underlying space into finite and infinite-dimensional subspaces. Using suitable spectral inequalities to verify the geometric conditions required by the linking theorem. Due to the critical growth, the Palais-Smale condition does not hold globally in both cases. Therefore, a key step is to show that the minimax level avoids the non-compactness levels associated with concentrating sequences $u_\varepsilon$. In the Carnot group setting, this is particularly delicate because the explicit form of the concentrating functions is not known. We overcome this difficulty by exploiting their asymptotic behavior, which is sufficient to derive the required energy estimates. Another difficulty comes from the presence of the first solution $u_t$ in the critical term, which introduces additional lower-order terms in the energy estimates. To handle this issue, we localize the critical nonlinearity on suitable subsets of $\Omega$ where the concentrating functions are sufficiently large. On these sets, the critical term dominates the perturbation induced by 
	$u_t$, allowing us to control the interaction terms. With this approach, we show that Palais-Smale sequences at the minimax level are bounded and converge to nontrivial limits. As a result, we obtain the existence of a second solution for homogeneous dimension $Q>6$, both in the mountain-pass case and in the linking case. 
	
	The proof of Theorem \ref{thm1.2} deals with the resonant case $\lambda = \lambda_1$. Under suitable smallness and sign assumptions on the function $f$, we show that the associated energy functional is bounded below on an appropriate constraint. A direct minimization argument then yields a nontrivial critical point, which corresponds to a weak solution $u_t$ of problem \eqref{1.1}. 
	
	The proof of Theorem \ref{thm1.3} is based on a bifurcation analysis near higher eigenvalues of the sub-Laplacian. Starting from the branch of nonpositive solutions obtained in Theorem \ref{thm1.1}, we write solutions of \eqref{1.1} in the form $u= u_t(\lambda) +v$, which leads to an equivalent equation for the perturbation $v$. Using the spectral decomposition of the underlying space and a classical bifurcation theorem by Rabinowitz, we show that each eigenvalue $\lambda_k$, $k>1$, is a bifurcation point. Finally, an energy estimate shows that the bifurcating solutions can only occur for $\lambda < \lambda_k$, which determines the direction of bifurcation.
	
	The main features of this work can be summarized as follows.
	
	\begin{itemize}
		\item To the best of our knowledge, Ambrosetti-Prodi type problems have not been previously investigated in the framework of Carnot groups, including the Heisenberg group. This paper provides the first existence, multiplicity, and bifurcation results in this non-Euclidean setting.
		
		\item The analysis is carried out for the sub-Laplacian on Carnot groups, where the operator is degenerate elliptic and many classical tools from elliptic theory are no longer available. The approach relies instead on subelliptic variational methods and the functional framework of Folland-Stein Sobolev spaces.
		
		\item The presence of critical Sobolev growth is treated by adapting Br\'ezis-Nirenberg type compactness arguments to the Carnot group setting. In contrast with the Euclidean case and the Heisenberg group, explicit extremal functions for the critical Sobolev inequality are not known in general Carnot groups, which requires a refined analysis.
		
	\end{itemize}
	
	The paper is organized as follows. In Section~\ref{sec1}, we recall basic notions on Carnot groups, the sub-Laplacian, and the functional framework, including spectral properties and Sobolev-type embeddings. Section~\ref{sec2} establishes the existence of a nonpositive solution and a second solution to problem \eqref{1.1}, using mountain-pass or linking arguments, with the critical growth handled via Br\'ezis-Nirenberg-type estimates adapted to the Carnot group setting. Section~\ref{sec3} deals with the resonant case $\lambda=\lambda_1$, where existence is proved via a direct minimization argument under suitable smallness and sign conditions on $f$. Finally, in Section~\ref{sec4}, we prove the bifurcation near higher eigenvalues.
	
	\section{Preliminaries}\label{sec1}
	The purpose of this section is to recall basic definitions and properties of Carnot groups that will be used throughout the paper. We follow the approach based on homogeneous Lie group structures, which is equivalent to the classical theory of stratified Lie groups.
	\subsection{Carnot groups and Folland-Stein spaces}
	Let $(\mathbb{R}^N,\circ)$ be a Lie group, where $\circ$ denotes the group law.  
	Assume that $(\mathbb{R}^N,\circ)$ is endowed with a family of Lie group automorphisms 
	$\{\delta_\mu\}_{\mu>0}$, called dilations, of the form
	\begin{equation}\label{dilation}
		\delta_\mu(\xi^{(1)},\xi^{(2)},\ldots,\xi^{(r)})
		=
		\bigl(\mu\, \xi^{(1)},\, \mu^2 \xi^{(2)},\, \ldots,\, \mu^r \xi^{(r)}\bigr),
	\end{equation}
	where $\xi^{(i)}\in \mathbb{R}^{N_i}$ for $i=1,\dots,r$ and $N_1+\cdots+N_r=N$. Let $\mathfrak{g}$ be the Lie algebra of $(\mathbb{R}^N,\circ)$, i.e., the space of all left-invariant vector fields. Let $X_1,\dots,X_{N_1}$ be the left-invariant vector fields  on $\G$ satisfying $X_j(0) = \partial/ \partial \xi_j^{(1)} \big|_0$ for $j=1,\dots, N_1$. We assume that these vector fields generate the whole Lie algebra under Lie brackets, i.e.
	$\operatorname{Lie}\{X_1,\dots,X_{N_1}\} = \mathfrak{g}.$
	Under this assumption, we refer to 
	$\mathbb{G} = (\mathbb{R}^N,\circ,\delta_\mu)$
	as a (homogeneous) Carnot group of step $r$, with $N_1$ generators. An equivalent description of a Carnot group is given in terms of the stratification of its Lie algebra. Precisely, a Carnot group can equivalently be defined as a connected, simply connected Lie group whose Lie algebra $\mathfrak{g}$ admits a stratification, i.e., a direct sum decomposition
	\begin{align}
		\label{stratification}
		\mathfrak{g} &= V_1 \oplus V_2 \oplus \cdots \oplus V_r, \quad \text{ satisfying } 
		&\left\{
		\begin{aligned}
			&[V_1, V_j] = V_{j+1}, &&\text{for } 1 \le j \le r-1, \\
			&[V_1, V_r] = \{0\}.
		\end{aligned}
		\right.
	\end{align}
	Here $V_1$ is called the horizontal layer, and $V_r\neq\{0\}$.
	The formulations \eqref{dilation} and \eqref{stratification} are equivalent. In particular,
	every stratified group admits homogeneous dilations of the form \eqref{dilation}, and conversely, every homogeneous Carnot group is isomorphic to a stratified group (see \cite{Bonfiglioli}). The homogeneous dimension of $\mathbb{G}$ is defined by
	$Q \coloneq  \sum_{j=1}^r j\, N_j.$ The sub-Laplacian associated with $\mathbb{G}$ is the second-order differential operator
	\[
	\Delta_{\mathbb{G}} = \sum_{i=1}^{N_1} X_i^2.
	\]   
	A basic example of a Carnot group is the abelian group $\mathbb{G} = (\mathbb{R}^N,+)$. In this case, the stratification is trivial, the homogeneous dimension as $Q=N$, and the associated sub-Laplacian coincides with the usual constant-coefficient elliptic operator on $\mathbb{R}^N$. A fundamental non-commutative example is the Heisenberg group $\mathbb{H}^n = (\mathbb{R}^{2n+1},\circ)$, whose homogeneous dimension is $Q=2n+2$. The group operation is given by $\xi \circ \xi' = \left( x+x',\ y+y',\ t+t' + 2\bigl( x'\cdot y - x \cdot y'\bigr)\right),$ for all $\xi= (x,y,t), \xi'= (x',y',t')\in\mathbb{R}^{2n+1}$, where $x,y,x',y'\in\mathbb{R}^n$ and $t,t'\in\mathbb{R}$.
	
	Now, we collect some basic analytic facts that will be used throughout the paper. Since the horizontal vector fields $X_1,\dots, X_{N_1}$ generate the entire Lie algebra $\mathfrak{g}$ under Lie brackets, they satisfy H\"ormander's condition. Hence, the sub-Laplacian $\Delta_{\mathbb{G}}$ is a hypoelliptic operator. Moreover, $\Delta_\G$ satisfies Bony’s maximum principle (see \cite{Bony}). The vector fields $X_1,\dots,X_{N_1}$ are homogeneous of degree one with respect to the family of dilations $\{\delta_\mu\}_{\mu>0}$ introduced in
	\eqref{dilation}. Moreover, the adjoint operator of $X_i$ is $-X_i$ in $L^2(\mathbb{G})$, and therefore the sub-Laplacian $\Delta_{\mathbb{G}}$ is a self-adjoint operator in divergence form. The Lebesgue measure on $\mathbb{R}^N$ coincides with the Haar measure of $\mathbb{G}$ (In the present work, all integrals are understood with respect to the Haar measure). In particular, it is left-invariant under group translations, and it satisfies the scaling property $|\delta_\mu(E)| = \mu^{Q}\, |E|$ for every measurable $E\subset\mathbb{R}^N$. For a smooth function $u: \G \to \R$, we denote by
	$\nabla_{\mathbb{G}}u \coloneq (X_1 u,\dots,X_{N_1} u)$
	the horizontal (or sub-elliptic) gradient associated with the operator $\Delta_{\mathbb{G}}$.  
	Both $\nabla_{\mathbb{G}}$ and $\Delta_{\mathbb{G}}$ are left-translation invariant and they are homogeneous of degree one and two respectively, with respect to $\delta_\mu$. We recall that a function $\rho: \G\to[0,\infty)$ is called a homogeneous (quasi)norm if
	\begin{itemize}
		\item $\rho(\xi)=0$ iff $\xi=0$;
		\item $\rho(\delta_\mu \xi)=\mu \ \rho(\xi)$ for all $\mu>0$;
		\item $\rho(\xi\circ \eta)\le C \bigl(\rho(\xi) + \rho(\eta)\bigr)$ for some $C\ge1$.
	\end{itemize}
	Such norms on every Carnot group always exist. A celebrated result of Folland \cite{Folland} ensures that
	there is a homogeneous norm $|\cdot|_\G$ and a constant $C_Q>0$ such that
	\begin{equation*}
		\Gamma_\xi(\eta)
		\coloneq
		\frac{C_Q}{|\xi^{-1}\circ \eta|_\G^{\ Q-2}},
		\qquad Q\ge3,
	\end{equation*}
	is a fundamental solution of $-\Delta_\G$ with pole at $\xi$. Homogeneous norms induce left-invariant distances of the form $d(\eta,\xi) = |\xi^{-1}\circ \eta|_\G,$ all of which are equivalent and generate the Euclidean topology on $\G$. Open balls with respect to the homogeneous distance $d$ are denoted by $B_r(\xi) \coloneq \{\eta \in \G : d(\eta,\xi) < r\}$.
	
	We now fix the notation and the function spaces for the problem \eqref{1.1}. Let $2^*_Q \coloneq \frac{2Q}{Q-2}$ denote the critical Sobolev exponent in the setting of Carnot groups.  The corresponding
	critical power in the nonlinearity is $2_Q^*-1 = \frac{Q+2}{Q-2}$, which plays the same
	role as the exponent $\frac{N+2}{N-2}$ in the classical semilinear Poisson equation in $\mathbb{R}^N, (N \geq 3)$. A fundamental role in the functional analysis on Carnot groups is played by the following
	Sobolev-type inequality (see, for instance, \cite{Folland}) 
	\begin{equation}\label{Sobolev}
		\|u\|_{L^{2^*_Q}(\Omega)}^2 \le C \|\nabla_\G u\|_{L^2(\Omega)}^2
		\quad \text{for all } u \in C_0^\infty(\Omega), 
	\end{equation}
	where $\Omega \subset \G$ is open and for some $C>0$. The optimal constant in \eqref{Sobolev} is explicitly known in some special cases, such as the Heisenberg group $\mathbb{H}^n$ (see \cite{Jerison2}).
	Let $\Omega \subset \G$ be an open set, the Folland-Stein Sobolev space associated with a Carnot group $\G$ is given by
	\[
	S^1(\Omega)
	=
	\bigl\{
	u \in L^{2_Q^*}(\Omega) \ : \ \nabla_\G u \in L^2(\Omega)
	\bigr\},
	\]
	and it is endowed with the norm
	\begin{equation}\label{FSnorm}
		\|u\|_{S^1(\Omega)} \coloneq \|u\|_{L^{2_Q^*(\Omega)}} + \|\nabla_\G u\|_{L^2(\Omega)} .
	\end{equation}
	We denote by $S_0^1(\Omega)$ the closure $C_0^\infty(\Omega)$ with respect to the norm \eqref{FSnorm}. By the Sobolev inequality \eqref{Sobolev}, the norm \eqref{FSnorm} is equivalent on $S_0^1(\Omega)$ to the norm induced by the inner product
	\[
	\langle u,v \rangle_{S_0^1(\Omega)}
	\coloneq
	\int_\Omega \langle \nabla_\G u, \nabla_\G v \rangle \, d\xi.
	\]
	As a consequence, $S_0^1(\Omega)$ is a Hilbert space. We emphasize that, for general unbounded
	domains $\Omega$, the space $S_0^1(\Omega)$ is not embedded into $L^2(\Omega)$. A function $u \in S_0^1(\Omega)$ is said to be a weak solution of the
	problem \eqref{1.1} if
	\begin{equation}\label{weak}
		\int_\Omega \langle \nabla_\G u, \nabla_\G \varphi \rangle \, d\xi
		= \lambda \int_\Omega u \ \varphi \ d\xi +
		\int_\Omega u^{2^*_Q-1}_+ \varphi \, d\xi + \int_\Omega f \ \varphi \ d\xi
		\quad \text{for all } \varphi \in S_0^1(\Omega).
	\end{equation}
	Every classical solution of \eqref{1.1} satisfies \eqref{weak}, since the formal adjoint of
	$\nabla_\G$ is $-\nabla_\G$. Finally, we point out that the exponent $2^*_Q$ is critical for the operator $-\Delta_\G$. Indeed, even when $\Omega$ is bounded, the continuous embedding
	\[
	S_0^1(\Omega) \hookrightarrow L^{2^*_Q}(\Omega)
	\]
	fails to be compact. This lack of compactness produces the main difficulty in the study of critical problems on Carnot groups. 
	
	We recall that if $Q \leq 3$, then the group $\mathbb{G}$ reduces to the Euclidean space $\mathbb{R}^N$ equipped with the usual additive structure. Therefore, throughout this paper we assume that the homogeneous dimension satisfies $Q > 3$.

	\subsection{Spectral properties of the sub-Laplacian on Carnot groups}
	Let $\Omega \subset \mathbb G$ be an open
	and a bounded domain. We consider the Dirichlet eigenvalue problem associated with the sub-Laplacian
	$\Delta_{\mathbb G}$:
	\begin{equation*}
		\left\{
		\begin{aligned}
			-\Delta_{\mathbb G} u &= \lambda u \quad &&\text{in } \Omega,\\
			u &= 0 \quad &&\text{on } \partial\Omega.
		\end{aligned}
		\right.
	\end{equation*}
	Since $\Omega$ is bounded, the embedding
	$\h \hookrightarrow L^2(\Omega)$ is compact. As a consequence, the operator $-\Delta_{\mathbb G}$ endowed with
	Dirichlet boundary conditions is self-adjoint on $L^2(\Omega)$ and has a compact resolvent. Therefore, its spectrum consists of a sequence of real eigenvalues
	\[
	0 < \lambda_1 \leq \lambda_2 \le \cdots \le \lambda_k \le \lambda_{k+1} \le \cdots,
	\qquad \lambda_k \to +\infty \quad \text{as } k \to \infty,
	\]
	each having finite multiplicity. The first eigenvalue $\lambda_1$ admits the variational characterization
	\[
	\lambda_1
	= \min_{u \in \h \setminus \{0\}}
	\frac{\|\nabla_{\mathbb G}u\|^2_{L^2(\Omega)}}
	{\|u\|^2_{L^2(\Omega)}}.
	\]
	The minimum is achieved by a function $e_1 \in \h$,
	which can be chosen to be strictly positive in $\Omega$.
	The positivity and the simplicity of $\lambda_1$
	follow from Bony’s maximum principle \cite{Bony} for degenerate operators. More generally, for each $k \in \mathbb N$, the $(k+1)$-th eigenvalue can be
	characterized by the min-max principle as
	\begin{equation}
		\label{2.00}
		\lambda_{k+1}
		= \min_{u \in \mathcal{S}^+ \setminus \{0\}}
		\frac{\|\nabla_{\mathbb G}u\|^2_{L^2(\Omega)}}
		{\|u\|^2_{L^2(\Omega)}},
	\end{equation}
	where
	\[
	\mathcal{S}^+
	\coloneq \bigl\{ u \in \h :
	\langle u, e_j\rangle_{\h} = 0,
	\quad j=1,\dots,k \bigr\}.
	\]
	For each $k \in \mathbb N$, there exists an eigenfunction
	$e_{k+1} \in \mathcal{S}^+$ associated with $\lambda_{k+1}$,
	which attains the minimum in \eqref{2.00}. Finally, the family $\{e_k\}_{k \in \mathbb N}$ forms an orthonormal basis of $L^2(\Omega)$ and an orthogonal basis of $\h$.
	Eigenfunctions of the Dirichlet sub-Laplacian are of class $C^\infty(\Omega)$,
	while their regularity up to $\partial\Omega$ depends on the geometry of the boundary.
	\subsection{Some asymptotic estimates} 
	Garofalo and Vassilev \cite{Garofalo} extended the classical concentration-compactness principle of Lions to the framework of Carnot groups. In particular, they showed that the optimal constant in the critical Sobolev embedding
	\[
	S_0^1(\G) \hookrightarrow L^{2_Q^*}(\G)
	\]
	can be characterized by the variational problem
	\begin{equation}\label{1.1.1}
		S_\G
		\coloneq
		\inf_{u \in S_0^1(\G) \setminus \{0\}}
		\frac{\|\nabla_\G u\|_{L^2(\G)}^2}
		{\|u\|_{L^{2_Q^*}(\G)}^2}.
	\end{equation}
	Moreover, every minimizing sequence for \eqref{1.1.1} is relatively compact in $S_0^1(\G)$, up to the action of left translations and homogeneous dilations on $\G$.
	As a consequence, the infimum in \eqref{1.1.1} is achieved. The attainability of $S_\G$ has direct implications for the associated critical problem. In particular, the equation
	\begin{equation}\label{criticalPDE}
		-\Delta_\G u = u^{2_Q^*-1} \quad \text{in } \G
	\end{equation}
	admits a nontrivial, non-negative solution
	$u \in S_0^1(\G).$ By Bony's maximum principle \cite{Bony}, every non-negative solution of \eqref{criticalPDE} is positive. Bonfiglioli and Uguzzoni in \cite{BonfiglioliUguzzoni} obtained asymptotic estimates at infinity of u in terms of the fundamental solution of $\Delta_\G$. More precisely, they proved that there exists $M>0$ such that
	\begin{equation*}
		u(\xi) \leq M \min\{1, d(\xi)^{2-Q}\}.
	\end{equation*}
	Moreover, there exists $\alpha \in (0,1)$ such that every positive solution $u$ of \eqref{criticalPDE} belongs to $\Gamma^\alpha(\G)$. Here, $\Gamma^\alpha(\G)$ denotes the Lipschitz space adapted to the homogeneous structure of the Carnot group $\G$, introduced by Folland \cite{Folland}. These
	spaces are defined using the group law and homogeneous norm $d$ on $\G$.
	In addition, Loiudice \cite{Loiudice2} proved that if
	$u \in S_0^1(\G)$ is a positive solution of \eqref{criticalPDE}, then there
	exists a constant $C>0$ such that
	\[
	u(\xi) \approx \frac{C}{d(\xi)^{Q-2}}
	\quad \text{as } d(\xi) \to \infty .
	\]
	Now, let $U$ be a fixed minimizer of \eqref{1.1.1} and for $\varepsilon > 0$, consider the rescaled family
	\[
	U_\varepsilon(\xi) = \varepsilon^{\frac{2-Q}{2}} \, U\left(\delta_{1/\varepsilon} \xi\right).
	\]
	Up to multiplicative constants, the functions $U_\varepsilon$ are solutions of
	$- \Delta_\G u = u^{2^*_Q-1} \quad \text{in } \G.$
	Moreover, they satisfy
	\[
	\|\nabla_\G U_\varepsilon\|_{L^2(\G)}^2 
	= \|U_\varepsilon\|_{L^{2^*_Q}(\G)}^{2^*_Q} 
	= S_\G^{Q/2}, \qquad \forall \, \varepsilon > 0.
	\]
	Let $R > 0$ be such that $B_d(0,R) \subset \Omega$ (we may assume $0 \in \Omega$). 
	Take $\varphi \in C_0^\infty(B_d(0,R))$ with $0 \leq \varphi \leq 1$ and 
	$\varphi \equiv 1$ in $B_d(0,R/2)$. Set
	\[
	u_\varepsilon(\xi) \coloneq \varphi(\xi) \ U_\varepsilon(\xi)\ \text{ for } \ \xi \in \G. 
	\]
	We now derive some estimates for the family $\{u_\varepsilon\}_{\varepsilon>0}$,
	which will play a crucial role in showing that the minimax level of the associated energy functional lies below a suitable threshold. To this end, we recall the following result.
	\begin{lem}\label{lem2}
		As $\varepsilon \to 0$, the functions $u_\varepsilon$ satisfy
		\begin{align*}
			\|\nabla_\G u_\varepsilon\|_{L^2(\Omega)}^2
			&= S_\G^{\frac{Q}{2}} + O(\varepsilon^{Q-2}),\\
			\|u_\varepsilon\|_{L^{2_Q^*}(\Omega)}^{2_Q^*}
			&= S_\G^{\frac{Q}{2}} + O(\varepsilon^{Q}),\\
			\|u_\varepsilon\|_{L^2(\Omega)}^2
			&\ge
			\begin{cases}
				C_1\, \varepsilon^2 + O(\varepsilon^{Q-2}), & \text{if } Q > 4, \\
				C_1\, \varepsilon^2 |\log \varepsilon| + O(\varepsilon^2), & \text{if } Q = 4,
			\end{cases}
		\end{align*}
		where $C_1>0$ is a constant independent of $\varepsilon$.
	\end{lem}
	For the proof, see \cite[Lemma 3.3]{Loiudice2}. Using the functions $u_\varepsilon$ and the asymptotic estimates in Lemma \ref{lem2}, we derive further estimates for $u_\varepsilon$.
	\begin{lem}
		\label{lem}
		For some positive constants $C_2$ and $C_3$ and $\varepsilon \to 0$, we have
		\begin{align*}
			\|u_\varepsilon\|_{L^1(\Omega)} \leq C_2 \varepsilon^{\frac{Q-2}{2}}, \quad
			\|u_\varepsilon\|_{L^{\2-1}(\Omega)}^{\2-1} \leq C_3 \varepsilon^{\frac{Q-2}{2}}.
		\end{align*}
	\end{lem}
	\begin{proof}
		We recall the polar coordinates formula for radial functions. For $0 \leq r_1 < r_2$ and any measurable function $f:[r_1,r_2]\to \mathbb{R}$,  
		\begin{equation}
			\label{2.10}
			\int_{B_d(0,r_2)\setminus B_d(0,r_1)} f(d(\xi)) \ d\xi
			= Q \ |B_d(0,1)| \int_{r_1}^{r_2} f(\rho)\ \rho^{Q-1} \ d\rho,
		\end{equation}
		whenever at least one of the integrals exists. Now, using the properties of $\varphi$, we have
		\begin{align*}
			\|u_\varepsilon\|_{L^1(\Omega)} & = \int_\Omega \varphi(\xi) \ U_\varepsilon (\xi) \ d\xi \leq \int_{B_d(0,R)} U_\varepsilon (\xi) \ d\xi  = \varepsilon^{\frac{2-Q}{2}}  \int_{B_d(0,R)} U(\delta_{1/\varepsilon}\xi ) \ d\xi.
		\end{align*}
		By the change of variables $\eta = \delta_{1/\varepsilon} \xi$ and choosing $r_0< R/\varepsilon$, the above inequality reduces to evaluating integrals of functions 
		that depend only on the homogeneous distance $d(\cdot)$, i.e. 
		\begin{align*}
			\|u_\varepsilon\|_{L^1(\Omega)} &\leq \varepsilon^{\frac{Q+2}{2}}  \int_{B_d(0, R/ \varepsilon)} U(\eta) \ d\eta\\
			& = \varepsilon^{\frac{Q+2}{2}} \int_{B_d(0,r_0)} U(\eta) \ d\eta + \varepsilon^{\frac{Q+2}{2}}\int_{B_d(0, R/ \varepsilon) \setminus B_d(0,r_0) } U(\eta) \ d\eta \\
			& \leq \varepsilon^{\frac{Q+2}{2}} C_0 + C \varepsilon^{\frac{Q+2}{2}}\int_{B_d(0, R/ \varepsilon) \setminus B_d(0,r_0) } \frac{d\eta}{d(\eta)^{Q-2}}, 
		\end{align*}
		where $C_0, C>0$. Then applying \eqref{2.10} gives
		\begin{align*}
			\|u_\varepsilon\|_{L^1(\Omega)}  & \leq C \varepsilon^{\frac{Q+2}{2}} \left[  1 + \int_{r_0}^{R/ \varepsilon} \frac{1}{\rho^{Q-2}} \rho^{Q-1} \ d\rho \right] = C \varepsilon^{\frac{Q+2}{2}} \left[  1 + \int_{r_0}^{R/ \varepsilon}  \rho \ d\rho \right]. 
		\end{align*}
		By integrating the above equation, we get
		\[ \|u_\varepsilon\|_{L^1(\Omega)} \leq C \varepsilon^{\frac{Q+2}{2}} + C \varepsilon^{\frac{Q-2}{2}} \leq C_2 \varepsilon^{\frac{Q-2}{2}}.\]
		Similarly, we can compute
		\begin{align*}
			\|u_\varepsilon\|^{\2-1}_{L^{\2-1}(\Omega)} 
			\leq \int_{B_d(0,R)} U_\varepsilon (\xi)^{\2-1} \ d\xi  = \varepsilon^{-\frac{Q+2}{2}}  \int_{B_d(0,R)} U(\delta_{1/\varepsilon} \xi)^{\2-1} \ d\xi= \varepsilon^{\frac{Q-2}{2}}  \int_{B_d(0, R/ \varepsilon)} U(\eta)^{\2-1} \ d\eta.
		\end{align*}
		By using Theorem \ref{2.2} and \eqref{2.10}, we have
		\begin{align*}
			\|u_\varepsilon\|^{\2-1}_{L^{\2-1}(\Omega)} 
			& \leq \varepsilon^{\frac{Q-2}{2}} \int_{B_d(0,r_0)} U(\eta)^{\2-1} \ d\eta + C \varepsilon^{\frac{Q-2}{2}}\int_{B_d(0, R/ \varepsilon) \setminus B_d(0,r_0) } \frac{d\eta}{d(\eta)^{Q+2}}  \\
			& \leq C \varepsilon^{\frac{Q-2}{2}} \left[  1 + \int_{r_0}^{R/ \varepsilon} \frac{1}{\rho^{Q+2}} \rho^{Q-1} \ d\rho \right] = C \varepsilon^{\frac{Q-2}{2}} \left[  1 + \int_{r_0}^{R/ \varepsilon}  \rho^{-3} \ d\rho \right].
		\end{align*}
		A direct computation yields
		\[ \|u_\varepsilon\|^{\2-1}_{L^{\2-1}(\Omega)} \leq  C \varepsilon^{\frac{Q-2}{2}} + C \varepsilon^{\frac{Q+2}{2}} \leq  C_3 \varepsilon^{\frac{Q-2}{2}}, \]
		for some $C_3>0$. 
	\end{proof}

	\section{The proof of Theorem \ref{thm1.1}}
	\label{sec2}
	\subsection{Existence of a Nonpositive Solution}
	First, we show that problem \eqref{1.1} admits a nonpositive solution.
	If $u \le 0$ in $\Omega$, then $u_+=0$ and $u$ necessarily satisfies
	\begin{equation}
		\label{2.1}
		\begin{cases}
			-\Delta_\G u = \lambda u + t e_1 + h, & \text{ in } \Omega,\\
			u=0, & \text{ on } \partial \Omega.
		\end{cases}
	\end{equation}
	Consider first the auxiliary problem
	\begin{equation}
		\label{2.2}
		\begin{cases}
			-\Delta_\G u = \lambda u + h, & \text{ in } \Omega,\\
			u=0, & \text{ on } \partial \Omega.
		\end{cases}
	\end{equation}
	If $\lambda<\lambda_1$, problem \eqref{2.2} admits a unique solution.
	If $\lambda>\lambda_1$ and $\lambda\neq\lambda_k$ for all $k\in\mathbb N$, existence and uniqueness follow from the Fredholm alternative.
	If $\lambda=\lambda_k$ for some $k\in\mathbb N$, a solution exists provided
	$h\in\ker(-\Delta_\G-\lambda I)^\perp$, and it is unique in this orthogonal complement. Let $u_0\in\h$ denote a solution of \eqref{2.2}. Moreover, since $h\in L^\infty(\Omega)$, we have $u_0\in L^\infty(\Omega)$. Assume $\lambda\neq\lambda_1$ and let $u_t$ be a solution of \eqref{2.1}. Setting $w\coloneq u_t-u_0$, we obtain
	\[
	-\Delta_\G w = \lambda w + t e_1 \quad \text{in } \Omega, \qquad w=0 \text{ on } \partial\Omega.
	\]
	Since $e_1$ is an eigenfunction associated with $\lambda_1$, we have $w = \frac{t}{\lambda_1-\lambda} e_1,$
	and therefore $u_t = u_0 + \frac{t}{\lambda_1-\lambda} e_1.$
	Using the positivity of $e_1$ and the boundedness of $u_0$,
	we choose $t<0$ if $\lambda<\lambda_1$ and $t>0$ if $\lambda>\lambda_1$,
	with $|t|$ sufficiently large, so that $u_t\le0$ in $\Omega$.
	Hence, \eqref{1.1} admits a nonpositive solution.
	
	\subsection{Existence of a second solution}
	Now, we aim to identify a second solution of equation \eqref{1.1} in the form $u = v + u_t$, with $v$ satisfying  
	\begin{equation}
		\label{2.4}
		-\Delta_\G v = \lambda v + (v + u_t)^{2^*_Q - 1}_+ \quad \text{in } \Omega, 
		\quad v = 0 \ \text{on} \ \partial \Omega.
	\end{equation}
	To achieve this, we search for a critical point of the associated functional
	$J : \h \to \mathbb{R}, $
	defined by  
	\[
	J(v) \coloneq \frac{1}{2} \int_\Omega |\nabla_\G v|^2 \ d\xi
	- \frac{\lambda}{2} \int_\Omega v^2 \ d\xi
	- \frac{1}{2^*_Q} \int_\Omega (v + u_t)_+^{2^*_Q} \ d\xi.
	\]
	Let $\lambda \in (\lambda_k, \lambda_{k+1})$ for some $k \in \mathbb{N}$ (The case $0 < \lambda < \lambda_1$ is treated by the Mountain Pass Theorem as a special case this). We proceed to verify that $J$ satisfies the geometric conditions required by the linking theorem \cite[Theorem 5.1]{Mawhin}. For this, we consider the orthogonal decomposition  
	\[
	\h = \mathcal{S}^- \oplus \mathcal{S}^+, 
	\]
	where  
	\[
	\mathcal{S}^- \coloneq \mathrm{span}\{e_1,\dots,e_k\}, 
	\quad 
	\mathcal{S}^+ \coloneq (\mathcal{S}^-)^\perp 
	= \left\{ u \in \h : 
	\langle u, e_j \rangle_{\h} = 0, 
	\ \forall j=1,\dots,k \right\}.
	\]
	Note that for $0 < \lambda < \lambda_1$, we choose $\mathcal{S}^- = \emptyset$ and $\mathcal{S}^+ = \h$.
	Using the spectral properties of $-\Delta_\G$, we have the following inequalities:
	\[
	\|\nabla_\G v\|_{L^2(\Omega)}^2 \ge \lambda_{k+1} \|v\|_{L^2(\Omega)}^2, 
	\quad \forall v \in \mathcal{S}^+, 
	\qquad
	\|\nabla_\G v\|_{L^2(\Omega)}^2 \le \lambda_k \|v\|_{L^2(\Omega)}^2, 
	\quad \forall v \in \mathcal{S}^-.
	\]
	Next, we define the sets used to verify the linking geometry: 
	\[
	S_\rho \coloneq \partial B_d(0,\rho) \cap \mathcal{S}^+, 
	\]
	\[
	Q \coloneq \left\{ v = w + s z : w \in \mathcal{S}^-, 
	\ \|\nabla_\G w\|_{L^2(\Omega)} \leq r, 
	\ 0 \leq s
	\leq R \right\}, 
	\]
	where $z \in \mathcal{S}^+$, and the constants $0<\rho<R$ and $r>0$. In order to select a suitable $z \in \mathcal{S}^+$ and control the functional $J$ on $Q$, we consider the orthogonal projections of $u_\varepsilon$ as follows:
	
	Let $\pi_+$ and $\pi_-$ denote the orthogonal projections from $\h$ onto $\mathcal{S}^+$ and $\mathcal{S}^-$, respectively. Using these projections, we have the following estimates:
	
	\begin{lem} 
		\label{lem3}  
		There exists a constant $C>0$ such that for all sufficiently small $\varepsilon>0$, the following estimates hold:
		\begin{align*}
			\|\pi_+(u_\varepsilon)\|_{L^1(\Omega)} &\le C \varepsilon^{\frac{Q-2}{2}}, \\
			\|\pi_+(u_\varepsilon) \|_{L^{2^*_Q-1}(\Omega)}^{2^*_Q-1} &\le C \varepsilon^{\frac{Q-2}{2}},  \\
			\left|\int_\Omega \left(|\pi_+(u_\varepsilon) |^{2^*_Q} - |u_\varepsilon|^{2^*_Q}\right) \, d\xi\right| &\le C \varepsilon^{Q-2},  \\
			\left|\int_\Omega \left(|\nabla_\G u_\varepsilon|^2 - |\nabla_\G (\pi_+(u_\varepsilon))|^2 \right) \, d\xi\right| &\le C \varepsilon^{Q-2}. 
		\end{align*}
	\end{lem}
	
	\begin{proof}
		We write $u_\varepsilon = \pi_+(u_\varepsilon)  + \pi_-(u_\varepsilon)$. Since $\mathcal{S}^-$ is finite-dimensional, $\pi_-(u_\varepsilon)$ can be expressed as $\pi_-(u_\varepsilon) = \sum_{i=1}^k \beta_i e_i$, where $\beta_i = \int_\Omega u_\varepsilon e_i \, d\xi$.  Since the first eigenfunction associated with $\lambda_1$ is positive, we have $\pi_{-} u_{\varepsilon} \not\equiv 0$.
		By the equivalence of norms in finite-dimensional spaces and Lemma~\ref{lem}, we have
		\[
		\|\pi_-(u_\varepsilon)\|_{L^2(\Omega)}^2 = \sum_{i=1}^k \beta_i^2
		= \sum_{i=1}^k
		\left(
		\int_{\Omega} u_{\varepsilon} e_i \ d\xi
		\right)^2 \nonumber 
		\leq \left(
		\sum_{i=1}^k \|e_i\|_{L^\infty(\Omega)}^2
		\right)
		\|u_{\varepsilon}\|_{L^1(\Omega)}^2  \le C \|u_\varepsilon\|_{L^1(\Omega)}^2 \le C \varepsilon^{Q-2}, 
		\]
		and $\|\pi_-(u_\varepsilon)\|_{L^\infty(\Omega)} \le C \varepsilon^{(Q-2)/2}.$
		From the triangle inequality, Lemma~\ref{lem}, the $L^1$ and $L^{2^*_Q-1}$ estimates follow immediately 
		\[
		\|\pi_+(u_\varepsilon) \|_{L^1(\Omega)} \le \|u_\varepsilon\|_{L^1(\Omega)} + \|\pi_-(u_\varepsilon)\|_{L^1(\Omega)} \le C \varepsilon^{(Q-2)/2}, 
		\quad 
		\|\pi_+(u_\varepsilon) \|_{L^{2^*_Q-1}(\Omega)}^{2^*_Q-1} \le C \varepsilon^{(Q-2)/2}.
		\]
		Next, using the integral representation
		\[
		|\pi_+(u_\varepsilon)|^{2^*_Q} - |u_\varepsilon|^{2^*_Q} = \int_0^1 \frac{d}{d\theta} |u_\varepsilon - \theta \pi_-(u_\varepsilon)|^{2^*_Q} \, d\theta
		\]
		and standard inequalities for powers of sums, we obtain
		\begin{align*}
			\left|\int_\Omega \left(|\pi_+(u_\varepsilon)|^{2^*_Q} - |u_\varepsilon|^{2^*_Q}\right) d\xi\right| &
			\leq \2 
			\int_{0}^1 
			\int_{\Omega}
			\left| u_{\varepsilon} - \theta \pi_{-}u_{\varepsilon} \right|^{\2-1}
			\left| \pi_{-}u_{\varepsilon} \right| \ d\xi  \ d\theta \\ 
			& \leq \2 (\2-1)
			\int_{0}^1 
			\int_{\Omega}
			\left( |u_{\varepsilon}|^{\2-1} + \theta^{\2-1} |\pi_{-}u_{\varepsilon}|^{\2-1} \right)
			|\pi_{-}u_{\varepsilon}| \ d\xi \ d\theta   \\
			&\le C \left( \|u_\varepsilon\|_{L^{2^*_Q-1}(\Omega)}^{2^*_Q-1} \|\pi_-(u_\varepsilon)\|_{L^\infty(\Omega)} + \|\pi_-(u_\varepsilon)\|_{L^{2^*_Q}(\Omega)}^{2^*_Q} \right) \le C \varepsilon^{Q-2}.
		\end{align*}
		Finally, by orthogonality of the projections,
		\[
		\|\nabla_\G u_\varepsilon\|_{L^2(\Omega)}^2 - \|\nabla_\G (\pi_+(u_\varepsilon))\|_{L^2(\Omega)}^2 = \|\nabla_\G (\pi_-(u_\varepsilon))\|_{L^2(\Omega)}^2 \le C \varepsilon^{Q-2},
		\]
		which proves the gradient estimate.
	\end{proof}
	For a fixed constant $K>0$, define
	\[
	A_{\varepsilon,K}
	\coloneq \left\{ \xi \in \Omega : \pi_{+} (u_{\varepsilon}(\xi)) > K \right\}.
	\]
	We first observe that $A_{\varepsilon,K}$ contains a small ball centered at the origin.
	Indeed, by the definition of $u_\varepsilon$ and the estimate on $\|\pi_-(u_\varepsilon)\|_{L^\infty(\Omega)}$, we have
	\[
	\pi_{+} u_{\varepsilon}(0)
	= u_{\varepsilon}(0) - \pi_{-} u_{\varepsilon}(0)
	\ge C \varepsilon^{-(Q-2)/2} - C \varepsilon^{(Q-2)/2}.
	\]
	Hence $\pi_{+} u_{\varepsilon}(0) \to +\infty$ as $\varepsilon \to 0$.  
	Since both $u_\varepsilon$ and $\pi_-(u_\varepsilon)$ are continuous, $\pi_+(u_\varepsilon)$ is continuous as well.  
	Therefore, there exists $\delta>0$, independent of $\varepsilon$, such that
	\[
	B_d(0,\delta) \subset A_{\varepsilon,K}.
	\]
	
	\begin{lem}
		\label{lem4}
		As $\varepsilon \to 0$, the following estimates hold:
		\begin{align*}
			\int_{A_{\varepsilon,K}} |\pi_{+} u_{\varepsilon}|^{2^*_Q} \, d\xi
			&= \int_{\Omega} |u_{\varepsilon}|^{2^*_Q} \, d\xi
			+ O(\varepsilon^{Q-2}), \\[4pt]
			\int_{A_{\varepsilon,K}} |\pi_{+} u_{\varepsilon}|^{2^*_Q-1} \, d\xi
			&= \int_{\Omega} |u_{\varepsilon}|^{2^*_Q-1} \, d\xi
			+ O(\varepsilon^{\frac{Q+2}{2}}), \\[4pt]
			\int_{A_{\varepsilon,K}} |\pi_{+} u_{\varepsilon}| \, d\xi
			&= \int_{\Omega} |u_{\varepsilon}| \, d\xi
			+ O(\varepsilon^{\frac{Q-2}{2}}). 
		\end{align*}
	\end{lem}
	
	\begin{proof}
		Writing $u_\varepsilon = \pi_+(u_\varepsilon) + \pi_-(u_\varepsilon)$, for any $p \ge 1$ we have
		\[
		\int_{A_{\varepsilon,K}} |\pi_+(u_\varepsilon)|^p \, d\xi
		= \int_{\Omega} |u_\varepsilon|^p \, d\xi
		+ \int_{A_{\varepsilon,K}} \left(|\pi_+(u_\varepsilon)|^p - |u_\varepsilon|^p\right) d\xi
		- \int_{\Omega \setminus A_{\varepsilon,K}} |u_\varepsilon|^p \, d\xi.
		\]
		Using the identity
		\[
		|\pi_+(u_\varepsilon)|^p - |u_\varepsilon|^p
		= -p \int_0^1 |u_\varepsilon - t \pi_-(u_\varepsilon)|^{p-2}
		(u_\varepsilon - t \pi_-(u_\varepsilon))\, \pi_-(u_\varepsilon) \, dt,
		\]
		together with Hölder’s inequality and Lemma~\ref{lem3}, we obtain
		\[
		\left| \int_{A_{\varepsilon,K}}
		\left(|\pi_+(u_\varepsilon)|^p - |u_\varepsilon|^p\right) d\xi \right|
		=
		\begin{cases}
			O(\varepsilon^{Q-2}), & p = 2^*_Q, \\[4pt]
			O(\varepsilon^{\frac{Q+2}{2}}), & p = 2^*_Q - 1, \\[4pt]
			O(\varepsilon^{\frac{Q-2}{2}}), & p = 1.
		\end{cases}
		\]
		Moreover, since $B_d(0,\delta) \subset A_{\varepsilon,K}$, it follows that
		\[
		\int_{\Omega \setminus A_{\varepsilon,K}} |u_\varepsilon|^p \, d\xi
		\le \varepsilon^{Q - \frac{(Q-2)}{2}p}
		\int_{d(\eta) > \delta/\varepsilon} |U(\eta)|^p \, d\eta,
		\]
		which yields
		\[
		\int_{\Omega \setminus A_{\varepsilon,K}} |u_\varepsilon|^p \, d\xi
		=
		\begin{cases}
			O(\varepsilon^{Q}), & p = 2^*_Q, \\[4pt]
			O(\varepsilon^{\frac{Q+2}{2}}), & p = 2^*_Q - 1, \\[4pt]
			O(\varepsilon^{\frac{Q-2}{2}}), & p = 1.
		\end{cases}
		\]
		Combining the above estimates proves Lemma~\ref{lem4}.
	\end{proof}
	Before constructing the linking structure, we first verify that the functional $J$ is bounded from below on small spheres in $\mathcal{S}^+$. The following lemma ensures that $J$ attains a positive level on such a sphere, which will serve as one part of the linking geometry.
	\begin{lem}
		\label{lem1}
		There exist constants $\rho>0$ and $\kappa>0$ such that
		\[
		J(v) \geq \kappa, \quad \forall \, v \in S_\rho \coloneq \partial B_d(0,\rho) \cap \mathcal{S}^+.
		\]
	\end{lem}
	\begin{proof}
		Let $v \in \mathcal{S}^+$. Since $u_t \leq 0$ and $\lambda \in (\lambda_k, \lambda_{k+1})$, we have
		\[
		J(v) \ge \frac{1}{2} \int_\Omega |\nabla_\G v|^2 \, d\xi - \frac{\lambda}{2} \int_\Omega v^2 \, d\xi - \frac{1}{2^*_Q} \int_\Omega (v)_+^{2^*_Q} \, d\xi.
		\]
		Applying the spectral inequality $\int_\Omega v^2 \, d\xi \leq \frac{1}{\lambda_{k+1}} \int_\Omega |\nabla_\G v|^2 \, d\xi$ for $v \in \mathcal{S}^+$, we get
		\[
		J(v) \geq \frac{1}{2} \left(1 - \frac{\lambda}{\lambda_{k+1}}\right) \int_\Omega |\nabla_\G v|^2 \, d\xi - \frac{1}{2^*_Q} \int_\Omega (v)_+^{2^*_Q} \, d\xi.
		\]
		Further, the Sobolev inequality on the Carnot group gives
		\[
		\int_\Omega (v)_+^{2^*_Q} \, d\xi \leq S_\G^{-\frac{Q}{Q-2}} \left(\int_\Omega |\nabla_\G v|^2 \, d\xi \right)^{2^*_Q/2},
		\]
		so that
		\[
		J(v) \geq \frac{1}{2} \left(1 - \frac{\lambda}{\lambda_{k+1}}\right) \|\nabla_\G v\|_{L^2(\Omega)}^2 - \frac{1}{2^*_Q} S_\G^{-\frac{Q}{Q-2}} \|\nabla_\G v\|_{L^2(\Omega)}^{2^*_Q}.
		\]
		Define the auxiliary function
		\[
		\Phi(\rho) \coloneq \frac{1}{2} \left(1 - \frac{\lambda}{\lambda_{k+1}}\right) \rho^2 - \frac{1}{2^*_Q} S_\G^{-\frac{Q}{Q-2}} \rho^{2^*_Q}, \quad \rho \ge 0.
		\]
		Observe that $\Phi$ is positive near zero, vanishes at zero, and attains its maximum at
		\[
		\rho \coloneq S_\G^{\frac{Q}{4}} \left(1 - \frac{\lambda}{\lambda_{k+1}}\right)^{\frac{Q-2}{4}} > 0.
		\]
		Denote this maximum by
		\[
		\kappa \coloneq \Phi(\rho) = \frac{S_\G^{\frac{Q}{2}}}{Q} \left(1 - \frac{\lambda}{\lambda_{k+1}}\right)^{\frac{Q}{2}} > 0.
		\]
		Thus, for every $v \in S_\rho$, we conclude
		\[
		J(v) \geq \kappa,
		\]
		which proves the lemma.
	\end{proof}
	To establish the linking geometry, we show that the functional $J$ remains strictly below the level $\kappa$ on the boundary of $Q$. To this end, we choose $z \in \mathcal{S}^+$ as follows. Let $u_\varepsilon$ be the bubble introduced earlier and $\pi_+$ the orthogonal projection onto $\mathcal{S}^+$. Set
	\[
	z \coloneq \pi_+(u_\varepsilon),
	\]
	so that $z \in \mathcal{S}^+$. With this choice, the critical term dominates along large values of the parameter $s$, yielding a uniform upper bound for $J$ on $\partial Q$.
	
	\begin{lem}
		\label{lem7}
		There exist constants $r_0>0$, $R_0>0$, and $\varepsilon_0>0$ such that for all $r \ge r_0$, $R \ge R_0$, and $0<\varepsilon \le \varepsilon_0$, we have
		\[
		J(v) < \kappa, \quad \forall v \in \partial Q,
		\]
		where $\kappa>0$ is defined in Lemma \ref{lem1}.
	\end{lem}
	
	\begin{proof}
		We split the boundary $\partial Q$ into three parts for clarity:
		\[
		Q_1 \coloneq \{ w \in \mathcal{S}^- : \|\nabla_\G w\|_{L^2(\Omega)} \le r\}, \quad
		Q_2 \coloneq \{ w + s \ \pi_+(u_\varepsilon) : w \in \mathcal{S}^-, \|\nabla_\G w\| = r, 0 \le s \le R \}, 
		\]
		\[
		Q_3 \coloneq \{ w + R \ \pi_+(u_\varepsilon) : w \in \mathcal{S}^-, \|\nabla_\G w\| \le r \}.
		\]
		\vspace{2mm}
		\noindent \textbf{Case 1: $v \in Q_1$.}  \\
		In this case, $v$ lies purely in $\mathcal{S}^-$. By the characterization of eigenvalues $\lambda_k$ and the fact that $\lambda \in (\lambda_k, \lambda_{k+1})$, we have
		\[
		J(v) = \frac{1}{2} \left( 1 - \frac{\lambda}{\lambda_k}  \right) \|\nabla_\G v\|^2_{L^2(\Omega)} - \frac{1}{2^*_Q} \int_\Omega (v+u_t)_+^{2^*_Q} \ d\xi \le 0.
		\]
		Thus, for all $v \in Q_1$, $J(v) < \kappa$, and this case is complete.
		
		\vspace{2mm}
		\noindent \textbf{Case 2: $v \in Q_2$.}  \\
		Let $\delta^2_0 \coloneq \sup_{0<\varepsilon \le 1} \|\nabla_\G \pi_+(u_\varepsilon)\|^2_{L^2(\Omega)}$.  
		If $s < s_0 \coloneq \sqrt{2\kappa}/\delta_0$, then the quadratic part dominates and we immediately get
		\begin{align*}
			J(v) &\le \frac{1}{2} \left( 1 - \frac{\lambda}{\lambda_k}  \right) r^2 + \frac{s^2}{2} \|\nabla_\G \pi_+(u_\varepsilon)\|^2_{L^2(\Omega)} - \frac{1}{\2} \| (w+ s \pi_-(u_\varepsilon)(u_\varepsilon) u_\varepsilon +u_t)_+\|_{L^{\2}(\Omega)}^{\2} \\
			&\le \frac{s^2}{2} \|\nabla_\G \pi_+(u_\varepsilon) \|^2_{L^2(\Omega)} \le \frac{s^2 \delta^2_0 }{2} < \kappa.
		\end{align*}
		For $s \ge s_0$, we estimate the nonlinear term by localizing on a subset of $\Omega$ where $\pi_+(u_\varepsilon)$ is sufficiently large.  
		Let $w \in \mathcal{S}^-$. Since $\mathcal{S}^-$ is finite dimensional and $u_t \in L^\infty(\Omega)$, there exists a constant $K_0$, independent of $s$ and $\varepsilon$ such that
		\[
		K_0 \coloneq \sup \left\| \frac{w + u_t}{s} \right\|_{L^\infty(\Omega)} < \infty.
		\]
		Since $\pi_+(u_\varepsilon) (0) \to \infty$ as $\varepsilon \to 0$, there exists $\bar \varepsilon>0$ such that, for $0<\varepsilon\le \bar \varepsilon$, the set
		\[  A_{\varepsilon, K_0} \coloneq \{ \xi \in \Omega : \pi_+(u_\varepsilon (\xi)) > K_0 \}\]
		has a positive measure for all sufficiently small $\varepsilon$. We estimate the nonlinear term on $A_{\varepsilon, K_0}$. Since $\pi_+ (u_\varepsilon) + \frac{ w+ u_t}{s} >0$ on $A_{\varepsilon, K_0}$, by the Fundamental Theorem of Calculus, we have 
		\begin{align*}
			&\int_{A_{\varepsilon, K_0}} \left( \pi_+ (u_\varepsilon) + \frac{ w+ u_t}{s}\right)^{\2} \ d\xi
			- \int_{A_{\varepsilon, K_0}} |\pi_+(u_\varepsilon)|^{\2} \ d\xi
			- \int_{A_{\varepsilon, K_0}} \left| \frac{w+u_t}{s}\right|^{\2} \ d\xi \\
			&=
			\2 \int_0^1 \int_{A_{\varepsilon, K_0}}
			\left(
			\left| \frac{w+u_t}{s}+ \tau \ \pi_+(u_\varepsilon) \right|^{\2-2} \left(\frac{w+u_t}{s}+ \tau \  \pi_+(u_\varepsilon)\right)
			- |\tau \ \pi_+(u_\varepsilon)|^{\2-2} \tau \ \pi_+(u_\varepsilon)
			\right) \pi_+(u_\varepsilon) \ d\xi \ d\tau .
		\end{align*}
		Using the mean value theorem, there exists $\theta=\theta(\xi)\in(0,1)$ such that
		\begin{align*}
			&\left| \frac{ w+ u_t}{s}+ \tau \ \pi_+(u_\varepsilon)\right|^{\2-2} \left(\frac{ w+ u_t}{s} + \tau \ \pi_+(u_\varepsilon)\right)
			- \left|\tau \ \pi_+(u_\varepsilon)\right|^{\2-2} \tau \ \pi_+(u_\varepsilon)\\
			&= (\2-1)\ \left|\tau \ \pi_+(u_\varepsilon) + \theta \frac{ w+ u_t}{s} \right|^{\2-2} \frac{ w+ u_t}{s} .
		\end{align*}
		Hence,
		\begin{align*}
			&\left|
			\int_{A_{\varepsilon, K_0}} \left( \pi_+(u_\varepsilon) +\frac{ w+ u_t}{s}\right)^{\2} \ d\xi
			- \int_{A_{\varepsilon, K_0}} \left|\pi_+(u_\varepsilon) \right|^{\2} \ d\xi
			- \int_{A_{\varepsilon, K_0}} \left|\frac{ w+ u_t}{s}\right|^{\2} \ d\xi
			\right|\\
			&\le
			C \int_0^1 \int_{A_{\varepsilon, K_0}}
			\left|\tau \ \pi_+(u_\varepsilon) + \theta \frac{ w+ u_t}{s} \right|^{\2-2} \left|\pi_+(u_\varepsilon) \right| \left|\frac{ w+ u_t}{s}\right| \ d\xi \ d\tau \\
			&\le
			C \int_{A_{\varepsilon, K_0}}
			\left( \left|\pi_+(u_\varepsilon) \right|^{\2-1} \left|\frac{ w+ u_t}{s}\right| + \left|\pi_+(u_\varepsilon) \right| \left|\frac{ w+ u_t}{s}\right|^{\2-1} \right)\ d\xi .
		\end{align*}
		Applying this inequality and H\"older's inequality, we obtain
		\begin{align}
			\left\| \left(\pi_{+}(u_{\varepsilon}) + \frac{w + u_{t}}{s} \right)_+ \right\|_{L^{\2}(\Omega)}^{\2} 
			&\geq \left\| \pi_{+}(u_{\varepsilon}) + \frac{w + u_{t}}{s} \right\|_{L^{\2}(A_{\varepsilon, K_0})}^{\2} \nonumber\\
			&\geq \left\| \pi_{+}(u_{\varepsilon}) \right\|_{L^{\2}(A_{\varepsilon,K_0})}^{\2}
			+ \left\| \frac{w + u_{t}}{s} \right\|_{L^{\2}(A_{\varepsilon,K_0})}^{\2} \nonumber \\
			&\quad - C \int_{A_{\varepsilon, K_0}} \left( |\pi_{+}(u_{\varepsilon})|^{\2-1}\left|\frac{w+u_{t}}{s} \right|
			+ |\pi_{+}(u_{\varepsilon})| \left|\frac{w+u_{t}}{s} \right|^{\2-1} \right)\ d\xi \nonumber \\
			&\geq \left\| \pi_{+}(u_{\varepsilon}) \right\|_{L^{\2}(A_{\varepsilon, K_0})}^{\2}
			+ \left\| \frac{w + u_{t}}{s} \right\|_{L^{\2}(A_{\varepsilon, K_0})}^{\2} \nonumber \\
			&\quad - C\Big( \| \pi_{+}(u_{\varepsilon}) \|_{L^{\2-1}(A_{\varepsilon, K_0})}^{\2-1}
			+ \| \pi_{+}(u_{\varepsilon}) \|_{L^{1}(A_{\varepsilon, K_0})}\Big). \label{2.23}
		\end{align}
		Then, by applying Lemmas \ref{lem2},\ref{lem3}, \ref{lem4} and using \eqref{2.23}, we deduce that, for $\varepsilon > 0$ sufficiently small,  
		\begin{align*}
			J(v) 
			&\leq \frac{1}{2}\Big(1 - \frac{\lambda}{\lambda_k}\Big) r^2 + \frac{s^2}{2} 
			\left( \| \nabla_\G \pi_{+}(u_{\varepsilon}) \|^2_{L^2(\Omega)} -\lambda \|\pi_+ (u_\varepsilon)\|_{L^2(\Omega)}^2 \right)\\
			& \quad - \frac{s^{\2}}{\2} \left[ 
			\| \pi_{+}(u_{\varepsilon}) \|^{\2}_{L^{\2}(A_{\varepsilon, K_0})} + \left\|\frac{w+u_t}{s}\right\|_{L^{\2}(A_{\varepsilon,K_0})}^{\2} - C \Big(
			\| \pi_{+}(u_{\varepsilon}) \|_{L^{\2-1}(A_{\varepsilon, K_0})}^{\2-1}
			+ \| \pi_{+}(u_{\varepsilon}) \|_{L^1(A_{\varepsilon, K_0})}\Big) \right]\\
			& \leq \frac{1}{2}\Big(1 - \frac{\lambda}{\lambda_k}\Big) r^2 + \frac{s^2}{2} \left(S_{\G}^{\frac{Q}{2}} + O(\varepsilon^{Q-2}) \right) - \frac{s^{\2}}{\2} \left( \|u_\varepsilon\|_{L^{\2}(\Omega)}^{\2} + O(\varepsilon^{Q-2})\right)\\
			& \quad + C s^{\2}\left( \|u_\varepsilon\|_{L^{\2-1}(\Omega)}^{\2-1} + O(\varepsilon^{\frac{Q+2}{2}}) + \|u_\varepsilon\|_{L^1(\Omega)} + O(\varepsilon^{\frac{Q-2}{2}})\right)\\
			& \leq \frac{1}{2}\Big(1 - \frac{\lambda}{\lambda_k}\Big) r^2 + \frac{s^2}{2} \left(S_\G^{\frac{Q}{2}} + O(\varepsilon^{Q-2}) \right) - \frac{s^{\2}}{\2} \left( S_\G^{\frac{Q}{2}} + O(\varepsilon^{Q}) + O(\varepsilon^{Q-2})\right) + C s^{\2}\left( \varepsilon^{\frac{Q-2}{2}} +  O(\varepsilon^{\frac{Q+2}{2}})\right)\\
			& \leq \frac{1}{2}\Big(1 - \frac{\lambda}{\lambda_k}\Big) r^2 + \frac{s^2}{2} S_\G^{\frac{Q}{2}} - \frac{s^{\2}}{\2} S_\G^{\frac{Q}{2}} + C s^{\2} \varepsilon^{\frac{Q-2}{2}}\\
			& \coloneq \frac{1}{2}\Big(1 - \frac{\lambda}{\lambda_k}\Big) r^2 + \Phi_\varepsilon(s).
		\end{align*}
		Using the estimates from \eqref{phi}, this becomes
		\[
		J(v) \leq 
		\frac{1}{2}\Big(1 - \frac{ \lambda}{\lambda_k}\Big)r^2 + \frac{1}{Q} S_{\G}^{\frac{Q}{2}}
		+ O(\varepsilon^{\frac{Q-2}{2}}).
		\]
		Since $\frac{1}{2}(1 - \frac{\lambda}{\lambda_k})r^2 \to -\infty$ as $r \to \infty$, there exists $r > 0$ such that $J(v) < 0$ for $v \in Q_2$. 
		
	   \noindent \textbf{Case 3: $v \in Q_3$.}  \\
		Here $v = w + R \pi_+ u_\varepsilon$ with $w \in \mathcal{S}^- \cap \bar{B}_d(0,r)$.  
		Since $\mathcal{S}^- \cap \bar{B}_d(0,r)$ is finite-dimensional and $u_t \in L^\infty$, there exists $K>0$ such that $\|w + u_t\|_{L^\infty} \le K$.  
		Choose $0<\varepsilon_0< \bar \varepsilon$ small enough so that $\pi_+ (u_\varepsilon(0)) > 2K$ for $0< \varepsilon < \varepsilon_0$. Then, for sufficiently large $R > R_0$, the set
		\[
		\Big\{ \xi \in \Omega : \pi_+ (u_\varepsilon(\xi)) + \frac{w(\xi)+u_t(\xi)}{R} > 1 \Big\}
		\]
		has a positive measure. Again, similar to Case 2, the dominant term is $\pi_+ (u_\varepsilon)$, ensuring
		\[
		J(v) < 0 \le \kappa, \quad \forall v \in Q_3.
		\]
		Combining the three cases, we conclude that there exist $\varepsilon_0$ and $R_0 >0$ such that if $\varepsilon < \varepsilon_0$ and $R> R_0$, we have $J(v) < \kappa$, for all $v \in \partial Q$. This establishes the desired linking geometry.
	\end{proof}
	We consider the function
	\[
	\Phi_\varepsilon(s)
	\coloneq \frac{s^2}{2} S_\G^{\frac{Q}{2}}
	- \frac{s^{2^*_Q}}{2^*_Q} S_\G^{\frac{Q}{2}}
	+ C s^{2^*_Q} \varepsilon^{\frac{Q-2}{2}},
	\qquad s>0.
	\]
	Since $2^*_Q>2$, we have $\Phi_\varepsilon(s)\to -\infty$ as $s\to+\infty$ and
	$\Phi_\varepsilon(0)=0$, hence $\Phi_\varepsilon$ attains a maximum at some
	$s_\varepsilon>0$. A direct computation yields
	\[
	\Phi_\varepsilon'(s)
	= s S_\G^{\frac{Q}{2}}
	- s^{2^*_Q-1} S_\G^{\frac{Q}{2}}
	+ 2^*_Q C s^{2^*_Q-1}\varepsilon^{\frac{Q-2}{2}}.
	\]
	The critical point $s_\varepsilon$ satisfies
	\[
	s_\varepsilon^{2^*_Q-2}
	=
	\frac{S_\G^{\frac{Q}{2}}}
	{S_\G^{\frac{Q}{2}} - 2^*_Q C \varepsilon^{\frac{Q-2}{2}}}
	= 1 + O\!\left(\varepsilon^{\frac{Q-2}{2}}\right),
	\]
	which implies
	\[
	s_\varepsilon
	= 1 + O\!\left(\varepsilon^{\frac{Q-2}{2}}\right).
	\]
	Evaluating $\Phi_\varepsilon$ at $s_\varepsilon$, we obtain
	\begin{align*}
		\Phi_\varepsilon(s_\varepsilon)
		= \frac{s_\varepsilon^2}{2} S_\G^{\frac{Q}{2}}
		- \frac{s_\varepsilon^{2^*_Q}}{2^*_Q} S_\G^{\frac{Q}{2}}
		+ C s_\varepsilon^{2^*_Q} \varepsilon^{\frac{Q-2}{2}} = \left(\frac{1}{2}-\frac{1}{2^*_Q}\right) S_\G^{\frac{Q}{2}}
		+ O\!\left(\varepsilon^{\frac{Q-2}{2}}\right) = \frac{1}{Q} S_\G^{\frac{Q}{2}}
		+ O\!\left(\varepsilon^{\frac{Q-2}{2}}\right),
	\end{align*}
	where we used $\frac{1}{2}-\frac{1}{2^*_Q}=\frac{1}{Q}$.
	Therefore,
	\begin{equation}
		\label{phi}
		\max_{s>0} \Phi_\varepsilon(s)
		= \frac{1}{Q} S_\G^{\frac{Q}{2}}
		+ O\!\left(\varepsilon^{\frac{Q-2}{2}}\right).
	\end{equation}
	
	We now prove that the energy of the functional $J$ on the entire linking
	set ${Q}$ is strictly below the critical Sobolev level.
	This estimate will play a crucial role in the compactness analysis
	at the minimax level.
	
	\begin{lem}
		\label{lem8}
		Assume that $Q > 6$. Then
		\[
		\max_{v \in {Q}} J(v)
		< \frac{1}{Q} S_\G^{\frac{Q}{2}},
		\]
		where $S_\G$ denotes the best Sobolev constant.
	\end{lem}
	
	\begin{proof}
		Let $\varepsilon>0$ be fixed sufficiently small so that all estimates established in the previous lemma hold.
		Take $v=w+s\,\pi_+(u_\varepsilon)\in Q$, where $w\in\mathcal S^-$ and $s\ge0$. From the definition of $J$, we write
		\begin{align*}
			J(v)
			&=
			\frac12\int_\Omega\big(|\nabla_\G w|^2-\lambda w^2\big)\,d\xi
			+\frac{s^2}{2}\int_\Omega\big(|\nabla_\G \pi_+(u_\varepsilon)|^2-\lambda(\pi_+(u_\varepsilon))^2\big)\,d\xi \nonumber\\
			&\quad
			-\frac{1}{2^*_Q}\int_\Omega\big(w+s\,\pi_+(u_\varepsilon)+u_t\big)_+^{2^*_Q}\,d\xi .
		\end{align*}
		\noindent\textbf{Case 1:} $s< s_0$, where $s_0$ is defined in Lemma \ref{lem7}. Using the boundedness of $\|\nabla_\G \pi_+(u_\varepsilon)\|_{L^2(\Omega)}$, we obtain
		\[
		J(v)
		\le
		\frac{s^2}{2}\|\nabla_\G \pi_+(u_\varepsilon)\|_{L^2(\Omega)}^2
		\le
		\frac{s^2}{2}\delta_0^2
		< \kappa <
		\frac{1}{Q}S_\G^{\frac{Q}{2}} .
		\]
		\noindent\textbf{Case 2:} Assume that $s \ge s_0$.
		If $s \ge s_0$, using \eqref{2.23} and arguing as in Lemmas~\ref{lem2} and~\ref{lem4},
		we obtain
		\begin{equation}\label{case2-ineq}
			J(w + s\,\pi_{+}(u_{\varepsilon})) \leq 
			\frac{s^{2}}{2} \int_{\Omega}
			\big( |\nabla_\G \pi_{+}(u_{\varepsilon})|^{2}
			- \lambda (\pi_{+}(u_{\varepsilon}))^{2} \big)\, d\xi
			- \frac{s^{2^*_Q}}{2^*_Q}
			\int_{A_{\varepsilon, K_0}} (\pi_{+}(u_{\varepsilon}))^{2^*_Q}\, d\xi 
			+ C s^{2^*_Q} \varepsilon^{\frac{Q-2}{2}} .
		\end{equation}
		Proceeding as in \eqref{phi}, we deduce from \eqref{case2-ineq} that
		\[
		J(w + s\,\pi_{+}(u_{\varepsilon}))
		\leq \frac{1}{Q}
		\left( \int_{\Omega}
		\big( |\nabla_\G \pi_{+}(u_{\varepsilon})|^{2}
		- \lambda (\pi_{+}(u_{\varepsilon}))^{2} \big)\, d\xi \right)^{\frac{Q}{2}}
		\left( \int_{A_{\varepsilon,K_0}}
		(\pi_{+}(u_{\varepsilon}))^{2^*_Q}\, d\xi \right)^{-\frac{Q-2}{2}}
		+ O\!\left(\varepsilon^{\frac{Q-2}{2}}\right).
		\]
		Using the estimates from Lemmas \ref{lem2}, \ref{lem3} and \ref{lem4} on $\pi_{+}(u_{\varepsilon})$, we deduce 
		\begin{align}
			J(w + s \pi_{+}(u_{\varepsilon})) 
			& \leq \frac{1}{Q}\frac{ \left( S_\G^{\frac{Q}{2}}+ O(\varepsilon^{Q-2}) -\lambda C \varepsilon^2 + O(\varepsilon^{Q-2}) \right)^{\frac{Q}{2}} }{ \left( S_\G^{\frac{Q}{2}}+ O(\varepsilon^{Q}) + O(\varepsilon^{Q-2})\right)^{\frac{Q-2}{2}}} + O\left(\varepsilon^{\frac{Q-2}{2}}\right)\nonumber\\
			&\leq \frac{1}{Q} S_\G^{\frac{Q}{2}} 
			- \frac{1}{2} \lambda O(\varepsilon^2) +  O\left(\varepsilon^{\frac{Q-2}{2}}\right), \nonumber
		\end{align}
		where we have used $\|\pi_+(u_\varepsilon) \|^2_{L^2(\Omega)} = \|u_\varepsilon\|_{L^2(\Omega)} - \|\pi_-(u_\varepsilon)\|^2_{L^2(\Omega)}$ $\geq C\varepsilon^2 +O(\varepsilon^{Q-2})$, for $Q>6$. Finally, if $Q > 6$, i.e. $2 < \tfrac{Q-2}{2}$, the result follows by choosing $\varepsilon > 0$ sufficiently small.
	\end{proof}

	As a consequence of Lemmas \ref{lem1}, \ref{lem7}, and \ref{lem8}, the functional $J$ satisfies the geometric assumptions of the
	linking theorem without the Palais-Smale condition
	\cite[Theorem 5.1]{Mawhin}.
	Therefore, there exists a sequence $\{v_n\} \subset \h$ such that
	\begin{align}
		J(v_n) &= \frac{1}{2} \int_{\Omega} \big( |\nabla_\G v_n|^2 - \lambda v_n^2 \big)\, d\xi 
		- \frac{1}{\2} \int_{\Omega} (v_n + u_t)^{\2}_+ \, d\xi
		= c + o(1), \label{2.25} \\
		\langle J'(v_n), \phi \rangle_{\h} &= \int_{\Omega} (\nabla_\G v_n \cdot \nabla_\G \phi - \lambda v_n \phi)\ d\xi 
		- \int_{\Omega} (v_n + u_t)^{\2-1}_+ \phi \, d\xi 
		= \varepsilon_n\|\nabla_\G\phi\|_{L^2(\Omega)},  \label{2.26}
	\end{align}
	for all $\phi \in \h$, where the minimax level $c$ is given by
	\[
	c \coloneq \inf_{\gamma \in \Gamma} \max_{v \in Q} J(\gamma(v)),
	\]
	with $\Gamma = \{ \gamma \in C({Q}, \h): \gamma(v)=v, \text{ if } v \in \partial Q$\}.
	Moreover, by Lemma \ref{lem8}, the level $c$ satisfies
	$$c < \frac{1}{Q} S_\G^{\frac{Q}{2}}.$$
	
	\begin{lem}
		\label{lem9}
		The sequence $\{v_n\}$ is bounded in $\h$.
	\end{lem}
	\begin{proof}
		First we compute
		\begin{align}
			J(v_n)- \frac{1}{2} \langle J'(v_n), v_n \rangle_{\h} 
			&= \frac{1}{2} \int_{\Omega} (v_n + u_t)^{2^*-1}_{+} v_n \, d\xi
			- \frac{1}{\2} \int_{\Omega} (v_n + u_t)^{\2}_{+} \, d\xi \nonumber\\
			&= \frac{1}{2} \int_{\Omega} (v_n + u_t)^{\2-1}_{+} (v_n + u_t)\, d\xi
			- \frac{1}{2} \int_{\Omega} (v_n + u_t)^{\2-1}_{+} u_t \, d\xi \nonumber\\
			&\quad - \frac{1}{\2} \int_{\Omega} (v_n + u_t)^{\2}_{+} \, d\xi. \nonumber
		\end{align}
		Since $u_t \le 0$ and $(v_n+u_t)_+ \ge 0$, the second integral on the right-hand
		side is nonnegative. Hence,
		\begin{align}
			J(v_n)- \frac{1}{2} \langle J'(v_n), v_n \rangle_{\h}
			&\ge \frac{1}{2} \int_{\Omega} (v_n + u_t)^{\2}_{+} \, d\xi
			- \frac{1}{\2} \int_{\Omega} (v_n + u_t)^{\2}_{+} \, d\xi \nonumber\\
			&= \frac{1}{Q} \int_{\Omega} (v_n + u_t)^{\2}_{+} \, d\xi.
			\label{2.27}
		\end{align}
		Combining \eqref{2.25}, \eqref{2.26}, and \eqref{2.27}, we obtain
		\begin{equation}
			\frac{1}{Q} \int_{\Omega} (v_n + u_t)^{\2}_{+} \ d\xi
			\le J(v_n)- \frac{1}{2} \langle J'(v_n), v_n \rangle_{\h}
			\le c + \frac{\varepsilon_n}{2}\|\nabla_\G v_n\|_{L^2(\Omega)} + o(1),
			\label{2.28}
		\end{equation}
		where $\varepsilon_n \to 0$ as $n \to \infty$.
		Next, we decompose
		$v_n = v_n^{(1)} + v_n^{(2)},
		\quad v_n^{(1)} \in \mathcal{S}^{+},$ where $v_n^{(2)} \in \mathcal{S}^{-}.$
		Taking $v_n^{(1)}$ as a test function in \eqref{2.26} and using the variational
		characterization of $\lambda_{k+1}$, we have
		\begin{align}
			\langle J'(v_n), v_n^{(1)} \rangle_{\h}
			&= \|\nabla_\G v_n^{(1)}\|_{L^2(\Omega)}^{2}
			- \lambda \|v_n^{(1)}\|_{L^2(\Omega)}^{2}
			- \int_{\Omega} (v_n + u_t)^{\2-1}_{+} v_n^{(1)} \, d\xi \nonumber\\
			&\ge \left(1 - \frac{\lambda}{\lambda_{k+1}}\right)
			\|\nabla_\G v_n^{(1)}\|_{L^2(\Omega)}^{2}
			- \int_{\Omega} (v_n + u_t)^{\2-1}_{+} v_n^{(1)} \, d\xi.
			\label{2.29}
		\end{align}
		
		Using \eqref{2.26}, \eqref{2.28}, and \eqref{2.29}, together with Hölder’s and Young’s inequalities, we obtain
		\begin{align*}
			\left(1 - \frac{\lambda}{\lambda_{k+1}}\right)
			\|\nabla_\G v_n^{(1)}\|_{L^2(\Omega)}^{2}
			&\le \int_{\Omega} (v_n + u_t)^{\2-1}_{+} v_n^{(1)} \, d\xi
			+ \varepsilon_n \|\nabla_\G v_n^{(1)}\|_{L^2(\Omega)} \nonumber\\
			&\le \|v_n^{(1)}\|_{L^{\2}(\Omega)}
			\left(\int_{\Omega} (v_n + u_t)^{\2}_{+} \, d\xi\right)^{\frac{\2-1}{\2}}
			+ \varepsilon_n \|\nabla_\G v_n^{(1)}\|_{L^2(\Omega)} \nonumber\\
			&\le \varepsilon \|v_n^{(1)}\|_{L^{\2}(\Omega)}^{2}
			+ C_\varepsilon \left(\int_{\Omega} (v_n + u_t)^{\2}_{+} \, d\xi\right)^{\frac{2(\2-1)}{\2}}
			+ \varepsilon_n \|\nabla_\G v_n^{(1)}\|_{L^2(\Omega)} \nonumber\\
			&\le \varepsilon C \|\nabla_\G v_n^{(1)}\|_{L^2(\Omega)}^{2}
			+ C_\varepsilon \left( c \ Q+ \frac{\varepsilon_n \ Q}{2} \|\nabla_\G v_n \|_{L^2(\Omega)} + o(1) \right)^{\frac{2(\2-1)}{\2}} \nonumber\\
			& \quad +\varepsilon_n \|\nabla_\G v_n^{(1)}\|_{L^2(\Omega)} \nonumber\\
			& \le \varepsilon C \|\nabla_\G v_n^{(1)}\|_{L^2(\Omega)}^{2} + C_\varepsilon' + C\varepsilon_n \left(
			\|\nabla_\G v_n\|_{L^2(\Omega)}^{\frac{Q+2}{Q}}
			+ \|\nabla_\G v_n^{(1)}\|_{L^2(\Omega)} \right).
		\end{align*}
		Choosing $\varepsilon>0$ such that
		\[
		0<\varepsilon C < 1 - \frac{\lambda}{\lambda_{k+1}},
		\]
		we deduce
		\begin{equation}
			\label{2.30}
			\|\nabla_\G v_n^{(1)}\|_{L^2(\Omega)}^{2}
			\le C + \varepsilon_n C \left(
			\|\nabla_\G v_n\|_{L^2(\Omega)}^{(Q+2)/Q}
			+ \|\nabla_\G v_n^{(1)}\|_{L^2(\Omega)} \right).
		\end{equation}
		By an analogous argument applied to $v_n^{(2)} \in \mathcal{S}^{-}$, we obtain
		\begin{equation}
			\label{2.31}
			\|\nabla_\G v_n^{(2)}\|_{L^2(\Omega)}^{2}
			\le \bar C + \varepsilon_n C \left(
			\|\nabla_\G v_n\|_{L^2(\Omega)}^{(Q+2)/Q}
			+ \|\nabla_\G v_n^{(2)}\|_{L^2(\Omega)} \right).
		\end{equation}
		Adding \eqref{2.30} and \eqref{2.31}, we conclude that the sequence
		$\{v_n\}$ is bounded in $\h$.
	\end{proof}
	We now establish the compactness of Palais-Smale sequences below the critical level.
	\begin{lem}
		\label{lem10}
		Let $\{v_n\} \subset \h$ be a Palais-Smale sequence for $J$ at level
		$c < \frac{1}{Q} S_\G^{\frac{Q}{2}}$. Then, up to a subsequence,
		$v_n \to v$ strongly in $\h$.
	\end{lem}
	
	\begin{proof}
		By Lemma \ref{lem9}, the sequence $\{v_n\}$ is bounded in $\h$. Hence, up to a subsequence, there exists $v \in \h$ such that
		\begin{equation}\label{2.32}
			v_n \rightharpoonup v \ \text{in } \h, \qquad
			v_n \to v \ \text{in } L^q(\Omega), \ 2 \le q < 2_Q^{*}, \qquad
			v_n \to v \ \text{a.e. in } \Omega .
		\end{equation}
		Since $\{v_n\}$ is bounded in $\h$, by the continuous embedding $\h \hookrightarrow L^{2_Q^{*}}(\Omega)$ we deduce that the sequence
		$\{(v_n+u_t)^{2_Q^{*}-1}_{+}\}$ is bounded in
		$L^{\frac{2_Q^{*}}{2_Q^{*}-1}}(\Omega)$.
		Consequently, up to a subsequence, there holds
		\begin{equation}
			\label{con}
			(v_n+u_t)^{2_Q^{*}-1}_{+} 
			\rightharpoonup
			(v+u_t)^{2_Q^{*}-1}_{+} 
			\quad \text{in }
			L^{\frac{2_Q^{*}}{2_Q^{*}-1}}(\Omega).
		\end{equation}
		Passing to the limit in \eqref{2.26}, we obtain $J'(v)=0$, that is, $v$ is a weak solution of
		\begin{equation*}
			-\Delta_\G v = \lambda v + (v + u_t)^{\2-1}_{+}.
		\end{equation*}
		Since $u_t \leq 0$, we observe that
		\[ \int_\Omega (v+u_t)_+^{\2-1} v \ d\xi= \int_\Omega (v+u_t)_+^{\2-1}(v+u_t-u_t) \ d\xi \geq \int_\Omega (v+u_t)_+^{\2} \ d\xi.\]
		Recalling that $J'(v)= 0$, we conclude that
		\[
		J(v)=\left(\frac12-\frac{1}{2_Q^*}\right)
		\int_{\Omega}(v+u_t)^{2_Q^*}_+\,d\xi
		=\frac{1}{Q}\int_{\Omega}(v+u_t)^{2_Q^*}_+\,d\xi\ge 0.
		\] 
		Set $w_n \coloneq v_n - v$. By the Brezis-Lieb lemma, we have
		\begin{align}
			\int_{\Omega} |\nabla_\G v_n|^2 \, d\xi
			&=
			\int_{\Omega} |\nabla_\G w_n|^2 \, d\xi
			+ \int_{\Omega} |\nabla_\G v|^2 \, d\xi + o(1), \label{psbl1}\\
			\int_{\Omega} (v_n + u_t)^{2_Q^{*}}_{+} \, d\xi
			&=
			\int_{\Omega} (w_n)^{2_Q^{*}}_{+} \, d\xi
			+ \int_{\Omega} (v + u_t)^{2_Q^{*}}_{+} \, d\xi + o(1). \label{psbl2}
		\end{align}
		Therefore, combining \eqref{2.32}, \eqref{psbl1} and \eqref{psbl2} we obtain
		\begin{equation}\label{2.35}
			J(v_n) = J(v) + \frac{1}{2} \int_{\Omega} |\nabla_\G w_n|^2\ d\xi
			- \frac{1}{\2} \int_{\Omega} (w_n)^{\2}_{+}\ d\xi + o(1).
		\end{equation}
		It is easy to see that the following identity holds
		\begin{align}\label{2.34}
			\int_{\Omega} (v + u_t)^{\2-1}_{+} v\ d\xi
			= \int_{\Omega} (v + u_t)^{\2}_{+}\ d\xi
			- \int_{\Omega} (v + u_t)^{2_Q^{*}-1}_{+} u_t \ d\xi.
		\end{align}
		From \eqref{2.32} \eqref{psbl1} and \eqref{2.34}, we compute
		\begin{align}
			\langle J'(v_n), v_n \rangle_{\h}
			&= \int_{\Omega} \big( |\nabla_\G v_n|^2 - \lambda v_n^2 \big)\, d\xi
			- \int_{\Omega} (v_n+u_t)^{2_Q^{*}-1}_{+} v_n \, d\xi + o(1) \nonumber\\
			&= \int_{\Omega} |\nabla_\G w_n|^2 \, d\xi
			+ \int_{\Omega} |\nabla_\G v|^2 \, d\xi
			- \lambda \int_{\Omega} v^2 \, d\xi \nonumber\\
			&\quad - \int_{\Omega} (v_n+u_t)^{2_Q^{*}}_{+} \, d\xi
			+ \int_{\Omega} (v_n+u_t)^{2_Q^{*}-1}_{+} u_t \, d\xi
			+ o(1). \nonumber
		\end{align}
		Using $J'(v)=0$ and \eqref{2.34}, we rewrite the terms involving $v$ to obtain
		\begin{align}
			\langle J'(v_n), v_n \rangle_{\h}
			&= \int_{\Omega} |\nabla_\G w_n|^2 \, d\xi + \int_{\Omega} (v+u_t)^{2_Q^{*}}_{+} \, d\xi - \int_{\Omega} (v+u_t)^{2_Q^{*}-1}_{+} u_t \, d\xi
			\nonumber\\
			&\quad - \int_{\Omega} (v_n+u_t)^{2_Q^{*}}_{+} \, d\xi
			+ \int_{\Omega} (v_n+u_t)^{2_Q^{*}-1}_{+} u_t \, d\xi
			+ o(1). \nonumber
		\end{align}
		By the Br\'ezis-Lieb lemma, we have
		\[
		\int_{\Omega} (v_n+u_t)^{2_Q^{*}-1}_{+} u_t \, d\xi
		=
		\int_{\Omega} (v+u_t)^{2_Q^{*}-1}_{+} u_t \, d\xi
		+
		\int_{\Omega} (w_n)^{2_Q^{*}-1}_{+} u_t \, d\xi
		+ o(1).
		\]
		Substituting this identity into the previous expression and using \eqref{con}, \eqref{psbl2} yields
		\begin{align}
			\nonumber
			\langle J'(v_n), v_n \rangle_{\h}
			&= \int_{\Omega} |\nabla_\G w_n|^2 \, d\xi
			- \int_{\Omega} (w_n)^{2_Q^{*}}_{+} \, d\xi
			+ \int_{\Omega} (w_n)^{2_Q^{*}-1}_{+} u_t \, d\xi
			+ o(1).
		\end{align}
		Since $\langle J'(v_n), v_n \rangle_{\h} \to 0$ and
		\[
		\int_{\Omega} (w_n)^{2_Q^{*}-1}_{+} u_t \, d\xi \to 0
		\quad \text{as } n \to \infty,
		\]
		we conclude that
		\begin{equation}\label{2.36}
			\int_{\Omega} |\nabla_\G (w_n)|^2 \, d\xi
			=
			\int_{\Omega} (w_n)^{2_Q^{*}}_{+} \, d\xi
			+ o(1).
		\end{equation}
		We define
		\[
		\lim_{n\to\infty} \int_{\Omega} |\nabla_\G w_n|^2 \, d\xi = a \ge 0.
		\]
		\noindent\textbf{Case 1:} If $a=0$. In this case, $v_n \to v$ strongly in $\h$, and hence
		$c = J(v).$\\
		\noindent\textbf{Case 2:} If $a>0$. From \eqref{2.36} and the Sobolev inequality, we get
		\begin{align}
			\|\nabla_\G w_n\|_{L^2(\Omega)}^2 
			\geq S_\G \Big( \int_{\Omega} |w_n|^{\2} \ d\xi \Big)^{2/\2} \geq S_\G \Big( \int_{\Omega} (w_n)^{\2}_{+}\ d\xi \Big)^{2/\2} \notag = S_\G \Big( \int_{\Omega} |\nabla_\G w_n|^2 \ d\xi + o(1) \Big)^{2/\2}, \nonumber
		\end{align}
		which implies
		\begin{equation}\label{2.37}
			a \geq S_\G^{Q/2}.
		\end{equation}
		Combining \eqref{2.35}, \eqref{2.36},\eqref{2.37} and recall that $J(v)\ge 0$, we get
		\[
		c + o(1) = J(v_n) = J(v) + \left(\frac{1}{2}- \frac{1}{\2}\right) \|\nabla_\G w_n\|^2_{L^2(\Omega)} =\frac{a}{Q} \geq \frac{1}{Q} S_\G^{Q/2}.
		\]
		This contradicts Lemma \ref{lem8}. Therefore, $a= 0$.  
		Finally, we conclude that \[
		v_n \to v \quad \text{strongly in } \h .
		\].
	\end{proof}
	
	\begin{proof}[Proof of Theorem \ref{thm1.1}]
		As shown in Lemmas \ref{lem1}, \ref{lem7}, and \ref{lem8}, the functional
		$J$ satisfies the linking geometry of \cite[Theorem~5.1]{Mawhin}.
		Hence, there exists a Palais-Smale sequence $\{v_n\}\subset\h$ at some
		level $c < \frac{1}{Q}S_\G^{\frac{Q}{2}}$.
		By Lemma \ref{lem9}, the sequence $\{v_n\}$ is bounded in $\h$, and by
		Lemma \ref{lem10}, up to a subsequence,
		\[
		v_n \to v \quad \text{strongly in } \h .
		\]
		Therefore, $v$ is a critical point of $J$ and satisfies \eqref{2.4}. Moreover, $v\not\equiv 0$. Indeed, if $v=0$, then the strong convergence implies
		\[
		c=\lim_{n\to\infty}J(v_n)=J(0)=0,
		\]
		which contradicts $c>0$. Hence, $v$ is nontrivial.
		Finally, setting $u=v+u_t$, we obtain a nontrivial weak solution of
		\eqref{1.1}. This completes the proof.
	\end{proof}

	\section{The proof of Theorem \ref{thm1.2}}
	\label{sec3}
	We consider the resonant problem associated with the first eigenvalue
	\begin{equation}\label{3.1}
		\begin{cases}
			- \Delta_\G u = \lambda_{1} u + u^{\2-1}_{+} + f(\xi) & \text{in } \Omega, \\[2mm]
			u = 0 & \text{on } \partial\Omega,
		\end{cases}
	\end{equation}
	where $\lambda_{1}$ denotes the first eigenvalue of the sub-Laplacian $-\Delta_\G$
	on $\Omega$ with homogeneous Dirichlet boundary conditions.
	In the resonant case $\lambda=\lambda_{1}$, a necessary condition for the
	solvability of \eqref{3.1} is
	\begin{equation}\label{3.2}
		\int_{\Omega} f(\xi)\, e_{1}(\xi)\, d\xi < 0,
	\end{equation}
	where $e_{1}$ is the positive $L^{2}$-normalized eigenfunction corresponding to
	$\lambda_{1}$. Throughout this section, we further assume that $f$ is sufficiently small in $L^{2}(\Omega)$.
	Let $\mathcal{S}^{-} \coloneq \operatorname{span}\{e_{1}\}$ and
	$\mathcal{S}^{+} \coloneq (\mathcal{S}^{-})^{\perp}$ in $\h$. Then every
	$u \in \h$ admits a unique decomposition of the form
	\[
	u = t e_{1} + v,
	\quad t \in \mathbb{R}, \quad v \in \mathcal{S}^{+}.
	\]
	With respect to this decomposition, the energy functional
	$J : \h \to \mathbb{R}$ associated with \eqref{3.1} can be written as
	\[
	J(u)
	= \frac{1}{2}\int_{\Omega} |\nabla_\G v|^{2}\, d\xi
	- \lambda_{1} \int_{\Omega} v^{2}\, d\xi
	- \frac{1}{\2} \int_{\Omega} (v + t e_{1})^{\2}_{+}\, d\xi
	- \int_{\Omega} f(\xi)\,(v + t e_{1})\, d\xi,
	\]
	where $t = \int_{\Omega} u e_{1}\, d\xi$.
	
	We first study the behavior of the functional along the resonant direction associated with the first eigenfunction. This allows us to reduce the problem to a constrained variational problem on the orthogonal complement.
	\begin{lem}\label{lem3.1}
		For every fixed $v \in \mathcal{S}^{+}$, the functional $J$ is bounded above
		when restricted to the one-dimensional subspace $\mathcal{S}^{-}$.
	\end{lem}
	\begin{proof}
		Fix $v \in \mathcal{S}^{+}$ and define
		\begin{equation*}
			h(t) \coloneq J(v + t e_{1}), \qquad t \in \mathbb{R}.
		\end{equation*}
		We first observe that for $t < 0$ the positive part satisfies
		$(v + t e_{1})_{+} \le v_{+}$, and therefore
		\[
		h(t)
		\le \frac{1}{2}\int_{\Omega} \bigl(|\nabla_\G v|^{2} - \lambda_{1} v^{2}\bigr)\, d\xi
		+ \|f\|_{L^{2}(\Omega)}\, \|v\|_{L^{2}(\Omega)}.
		\]
		Hence $h$ is bounded above on $(-\infty,0]$.
		Next, we consider the behavior of $h(t)$ as $t \to +\infty$.
		We claim that
		\begin{equation}\label{3.4}
			\lim_{t \to +\infty}
			\left(
			\frac{1}{\2} \int_{\Omega} (v + t e_{1})^{\2}_{+}\, d\xi
			+ \int_{\Omega} f(\xi)\,(v + t e_{1})\, d\xi
			\right)
			= +\infty.
		\end{equation}
		It follows that
		$h(t) \to -\infty$ as $t \to +\infty$.
		To prove \eqref{3.4}, set
		\[
		K \coloneq \sup_{\xi \in \Omega} e_{1}(\xi).
		\]
		Since $e_{1}$ is continuous and positive in $\Omega$, there exists an open set
		$\Omega_{0} \subset \Omega$ of positive measure such that
		\[
		e_{1}(\xi) > \frac{K}{2} \qquad \text{for all } \xi \in \Omega_{0}.
		\]
		By Lusin’s theorem, for any $\varepsilon > 0$ (choose, for instance,
		$\varepsilon = |\Omega_{0}|/2$), there exists a continuous function
		$g : \Omega_{0} \to \mathbb{R}$ such that
		\[
		\big|\bigl\{\xi \in \Omega_{0} : g(\xi) \neq v(\xi)\bigr\} \big| < \varepsilon.
		\]
		Consequently, the set
		\[
		H \coloneq \bigl\{\xi \in \Omega_{0} : g(\xi) = v(\xi)\bigr\}
		\]
		has a positive measure. Since $g$ is continuous on $\Omega_{0}$, it is bounded on $H$. Define
		\[
		M \coloneq \sup_{\xi \in H} |v(\xi)|.
		\]
		Then, for all $\xi \in H$ and for every $t \ge t_{0} \coloneq 4M/K$, we have
		\[
		e_{1}(\xi) + \frac{v(\xi)}{t}
		\ge \frac{K}{2} - \frac{M}{t}
		\ge \frac{K}{4}.
		\]
		It follows that
		\[
		\left(e_{1} + \frac{v}{t}\right)_{+}^{\2} \ge \left(\frac{K}{4}\right)^{\2}
		\quad \text{a.e. in } H,
		\]
		and hence there exists $\delta > 0$ such that
		\[
		\int_{\Omega} \left(e_{1} + \frac{v}{t}\right)_{+}^{\2}\, d\xi
		\ge \delta
		\qquad \text{for all } t \ge t_{0}.
		\]
		Indeed,
		\[
		\int_{\Omega} (v + t e_1)^{\2}_{+}\, d\xi
		= t^{\2} \int_{\Omega} \left(e_1 + \frac{v}{t}\right)^{\2}_{+}\, d\xi
		\ge \delta\, t^{\2},
		\]
		for all $t \ge t_0$, which proves \eqref{3.4}. Since $h$ is continuous and bounded above on $(-\infty,0]$ and satisfies
		$h(t)\to -\infty$ as $t\to+\infty$, it follows that $J(v+te_1)$ is bounded above
		with respect to $t\in\mathbb R$.
		Consequently, the functional $J$ is bounded above on $\mathcal{S}^{-}$.
	\end{proof}
	By Lemma~\ref{lem3.1}, the map
	$t \longmapsto J(v + t e_1)$
	is bounded above on $\mathbb{R}$ and satisfies
	$h(t) \to -\infty \quad \text{as } t \to +\infty.$
	Hence $h$ attains a global maximum on $\mathbb{R}$.
	We now characterize this
	maximizer and its dependence on $v$.
	
	\begin{lem}\label{lem3.4}
		For every $v \in \mathcal{S}^{+}$, there exists a unique real number $t(v)$ such that
		\[
		J(v + t(v) e_1) = \max_{t \in \mathbb{R}} J(v + t e_1).
		\]
		Moreover, the mapping
		\[
		v \in \mathcal{S}^{+} \longmapsto t(v) \in \mathbb{R}
		\]
		is of class $C^1$.
	\end{lem}
	\begin{proof}
		Let $t_0$ be a global maximizer of the function $h(t)=J(v+te_1)$. Since $h$ is differentiable, we have
		\begin{equation}\label{3.6}
			h'(t_0)
			= - \int_{\Omega} (v+t_0e_1)_{+}^{\2-1} e_1\, d\xi
			- \int_{\Omega} f e_1\, d\xi
			= 0.
		\end{equation}
		Moreover, a direct computation shows that
		\[
		h''(t)
		= - \int_{\Omega} (v+te_1)_{+}^{\2-2} e_1^{2}\, d\xi
		\le 0,
		\]
		and therefore $h$ is concave on $\mathbb{R}$. Thus, the set of maxima is a closed interval. We now show that this interval reduces to a single point. 
		
		Assume by contradiction that the maximizer is not unique. Then $h$ is constant
		on an interval of maximizers, which implies that
		$h''(t_0)=0$ at any maximizer $t_0$. From the expression of $h''$, this yields $(v+t_0e_1)_{+}=0$ a. e. in $\Omega$. Substituting this information
		into \eqref{3.6}, we obtain
		\[
		\int_{\Omega} f e_1\, d\xi = 0,
		\]
		which contradicts assumption~\eqref{3.2}. Hence, the maximizer is unique, and we
		denote it by $t(v)$. Since $h'(t(v))=0$ and $h''(t(v))<0$, the Implicit Function Theorem ensures that
		the mapping $v \mapsto t(v)$ is of class $C^{1}$. 
	\end{proof}
	
	For the special case $v=0$, the maximizer $t(0)$ of $h(t)=J(te_1)$ satisfies
	\begin{equation*}
		\int_{\Omega} (t(0)e_1)_{+}^{\2-1} e_1 \, d\xi
		+ \int_{\Omega} f e_1 \, d\xi = 0.
	\end{equation*}
	Since $f$ satisfies condition \eqref{3.2}, we have $t(0) > 0$. Indeed, we can solve for $t(0)$ explicitly as
	\begin{equation*}
		t(0) = \left( \frac{- \int_{\Omega} f e_1 \, d\xi}{\int_{\Omega} e_1^{\2} \, d\xi} \right)^{\frac{1}{\2-1}}.
	\end{equation*}
	Using the map $v \mapsto t(v)$, we define the functional
	\begin{equation*}
		F : \mathcal{S}^{+} \longrightarrow \mathbb{R}, 
		\qquad F(v) \coloneq J(v + t(v)e_1),
	\end{equation*}
	whose minimization over a suitable ball in $\mathcal{S}^{+}$ will lead to a solution of the original problem. At $v=0$, the functional takes the value
	\begin{equation*}
		F(0) = J(t(0) e_1) = - \frac{1}{\2} \int_{\Omega} (t(0) e_1)^{\2} \, d\xi - t(0) \int_{\Omega} f e_1 \, d\xi
		= \frac{\2-1}{\2} t(0) \int_{\Omega} (-f e_1) \, d\xi > 0.
	\end{equation*}
	Putting the value of $t(0)$ in the above, we get
	\begin{equation}
		\label{3.13}
		F(0) = \frac{Q+2}{2Q} \frac{ \displaystyle \left( -\int_\Omega f\ e_1 \ d\xi \right)^{\frac{\2}{\2-1}}}{\left( \displaystyle \int_\Omega e_1 ^{\2} d\xi \right)^{\frac{1}{\2-1}}}.
	\end{equation}
	Next, for general $v \in \mathcal{S}^{+}$, the functional can be written as
	\begin{equation*}
		F(v) = \frac{1}{2}\int_{\Omega} (|\nabla_\G v|^{2} - \lambda_{1} v^{2})\ d\xi 
		- \frac{1}{\2}\int_{\Omega} (v + t(v)e_{1})^{\2}_{+}\ d\xi 
		- \int_{\Omega} f(v + t(v)e_{1})\ d\xi.
	\end{equation*}
	The purpose of introducing the functional $F$ is to obtain a critical point of $J$ by minimizing $F$ on $\mathcal{S}^{+}$. To this end, we show that, under suitable smallness assumptions on the function $f$, the functional $F$ admits a local minimum in the interior of a ball of $\mathcal{S}^{+}$.
	Let
	\begin{align}
		K_{1} &= \frac{\lambda_2^{-\frac{Q}{4}}}{Q+2} \ S_\G^{\frac{Q}{4}} 
		\left(\frac{Q}{Q+2}\right)^{\frac{Q-2}{4}} (\lambda_{2} - \lambda_{1})^{\frac{Q+2}{4}}, \label{3.15} \\
		K_{2} &= 
		\min \left\{
		\left(\frac{2}{Q+2}\right)^{\frac{Q+2}{2Q}} S_\G^{\frac{Q+2}{4}},\,
		\left(\frac{1}{Q+2}\right)^{\frac{Q+2}{2Q}} \|e_1\|_{\2}
		\left[\frac{Q}{Q+2}\left(1-\frac{\lambda_1}{\lambda_2}\right)S_\G \right]^{\frac{Q+2}{4}}
		\right\}. \label{3.16}
	\end{align}
	In addition to condition \eqref{3.2}, we assume that $f$ satisfies
	\begin{equation}\label{3.17}
		\|f\|_{L^2(\Omega)} \leq K_{1}, \qquad -\int_{\Omega} f e_{1}\ d\xi < K_{2}.
	\end{equation}
	\begin{lem}\label{lem3.2}
		Assume that conditions \eqref{3.2} and \eqref{3.17} are satisfied.
		Then there exists a constant $\alpha_0 > 0$ such that
		\begin{equation*}
			F(v) \ge \alpha_0 > F(0),
		\end{equation*}
		for every $v \in \mathcal{S}^{+}$ with $\|\nabla_\G v\|_{L^2(\Omega)} = \sigma_{0}$, where
		\[
		\sigma_{0}
		= \left[\frac{Q}{Q+2}\left(1 - \frac{\lambda_{1}}{\lambda_{2}}\right)\right]^{\frac{Q-2}{4}}
		S_\G^{\frac{Q}{4}}.
		\]
	\end{lem}
	\begin{proof}
		Let $v \in \mathcal{S}^{+}$. Since $t(v)$ maximizes $t \mapsto J(v+te_1)$, we have
		\[
		F(v) = J(v+t(v)e_1) \ge J(v).
		\]
		Using the spectral inequality
		\[
		\int_{\Omega} |\nabla_\G v|^{2}\, d\xi 
		\ge \lambda_{2} \int_{\Omega} v^{2}\, d\xi,
		\qquad \forall v \in \mathcal{S}^{+},
		\]
		we obtain
		\begin{align}
			F(v) &\geq J(v)
			= \frac12 \int_{\Omega} \big(|\nabla_\G v|^{2} - \lambda_{1} v^{2}\big)\, d\xi
			- \frac{1}{\2} \int_{\Omega} |v|^{\2}\, d\xi
			- \int_{\Omega} f v\, d\xi \notag\\
			&\ge \frac12\Big(1-\frac{\lambda_{1}}{\lambda_{2}}\Big)
			\int_{\Omega} |\nabla_\G v|^{2}\, d\xi
			- \frac{1}{\2} \int_{\Omega} |v|^{\2}\, d\xi
			- \|f\|_{2}\|v\|_{2}.
			\label{3.19}
		\end{align}
		By the Sobolev inequality
		\[
		\|v\|_{L^{\2}(\Omega)} \le S_\G^{-\frac{1}{2}} \|\nabla_\G v\|_{L^2(\Omega)},
		\]
		and the setting
		$\sigma \coloneq \|\nabla_\G v\|_{L^2(\Omega)}$, we deduce from
		\eqref{3.19} that
		\begin{equation}\nonumber
			F(v)
			\ge \Phi(\sigma)
		\end{equation}
		where
		\[
		\Phi(\sigma)
		=
		\frac12\Big(1-\frac{\lambda_1}{\lambda_2}\Big)\sigma^{2}
		-\frac{1}{\2} S_\G^{-\frac{Q}{Q-2}} \sigma^{\2}- \|f\|_{2}\lambda_{2}^{-1/2}\sigma,
		\qquad \sigma \ge 0.
		\]
		A direct computation shows that $\Phi$ attains its maximum at
		\[
		\sigma_0=
		\left[
		\frac{Q}{Q+2}
		\Big(1-\frac{\lambda_1}{\lambda_2}\Big)
		\right]^{\frac{Q-2}{4}}
		S_\G^{\frac{Q}{4}},
		\]
		Evaluating $\Phi$ at $\sigma_{0}$ yields
		\[
		\Phi(\sigma_{0})
		=
		\sigma_{0}
		\left[
		\frac{2}{Q+2} \Big(\frac{Q}{Q+2}\Big)^{\frac{Q-2}{4}}
		\Big(1-\frac{\lambda_{1}}{\lambda_{2}}\Big)^{\frac{Q+2}{4}} S_\G^{\frac{Q}{4}}
		- \|f\|_{2}\lambda_{2}^{-1/2}
		\right].
		\]
		By the smallness assumption $\|f\|_{L^2(\Omega)} \le K_1$ in \eqref{3.17} and using \eqref{3.15}, we have
		\begin{align*}
			\Phi(\sigma_{0})
			&\geq 
			\sigma_{0}
			\left[
			\frac{2}{Q+2} \Big(\frac{Q}{Q+2}\Big)^{\frac{Q-2}{4}}
			\Big(1-\frac{\lambda_{1}}{\lambda_{2}}\Big)^{\frac{Q+2}{4}} S_\G^{\frac{Q}{4}}
			- K_1 \lambda_{2}^{-1/2}
			\right] \\
			& \geq \sigma_0 \left[
			\frac{1}{Q+2} \Big(\frac{Q}{Q+2}\Big)^{\frac{Q-2}{4}}
			\Big(1-\frac{\lambda_{1}}{\lambda_{2}}\Big)^{\frac{Q+2}{4}} S_\G^{\frac{Q}{4}}
			\right] \\
			& =\frac{1}{Q+2} \Big( \frac{Q}{Q+2}\Big)^{\frac{Q-2}{2}} \Big( 1- \frac{\lambda_1}{\lambda_2}\Big)^{\frac{Q}{2}} S_\G^{\frac{Q}{2}} \coloneq \alpha_0 >0.
		\end{align*}
		From \eqref{3.19}, we obtain
		\[   F(v) \ge \alpha_0   \quad \text{ for all } v \in \mathcal{S}^+  \text{ with } \|\nabla_\G v \|_{L^2(\Omega)} = \sigma_0.\]
		Using the explicit expression of 
		$F(0)$ in \eqref{3.13} together with the second smallness condition in \eqref{3.17}, we obtain
		\begin{align*}
			F(0) \leq \frac{1}{2Q} \left[ \frac{Q}{Q+2} \Big( 1-\frac{\lambda_1}{\lambda_2}\Big) S_\G \right]^{\frac{Q}{2}}
		\end{align*}
		Consequently, under assumptions \eqref{3.2} and \eqref{3.17}, there exists a constant $\alpha_0>0$ such that
		\[
		F(v) \ge \alpha_0 > F(0),
		\qquad
		\forall\, v \in \mathcal{S}^{+}
		\text{ with } \|\nabla_\G v\|_{L^2(\Omega)}=\sigma_0.
		\]
		This completes the proof.
	\end{proof}
	From assumption \eqref{3.17} we also obtain the estimate
	\begin{equation}\label{3.23}
		F(0) < \frac{1}{Q} S_\G^{Q/2}.
	\end{equation}
	In view of Lemma~\ref{lem3.2}, the functional $F$ exhibits a local geometric
	structure on $\mathcal{S}^{+}$, which motivates the following constrained
	minimization problem:
	\begin{equation}\label{3.24}
		m \coloneq \min \bigl\{ F(v) : v \in B_{d}(0,\sigma_0) \bigr\}.
	\end{equation}
	
	\begin{lem}
		The minimization problem \eqref{3.24} admits a nontrivial solution
		$v_{0} \in B_{d}(0,\sigma_{0})$. Consequently, problem \eqref{3.1}
		possesses a weak solution.
	\end{lem}
	\begin{proof}
		From \eqref{3.23}, we have
		\begin{equation}\label{3.25}
			m < \frac{1}{Q} S_\G^{Q/2}.
		\end{equation}
		Let $\{v_n\} \subset B_d(0, \sigma_0)$ be a minimizing sequence for \eqref{3.24}. Since $\|\nabla_\G v_n\|_{L^2(\Omega)} \le \sigma_0$, up to a subsequence, there exists $v_0 \in \mathcal{S}^+$ such that
		\begin{equation}\label{3.26}
			v_n \rightharpoonup v_0 \text{ weakly in } \h, \quad
			v_n \to v_0 \text{ in } L^q(\Omega), \ 2 \le q < \2, \quad
			v_n \to v_0 \text{ a.e. in } \Omega.
		\end{equation}
		By weak lower semicontinuity,
		\begin{equation}\nonumber
			\|\nabla_\G v_0\|_{L^2(\Omega)} \le \liminf_{n\to\infty} \|\nabla_\G v_n\|_{L^2(\Omega)} \le \sigma_0.
		\end{equation}
		Using Ekeland’s variational principle, we may assume that
		\begin{equation}\nonumber
			F(v_n) \to m, \qquad F'(v_n) \to 0 \text{ as } n \to \infty.
		\end{equation}
		This implies
		\begin{align}
			\frac{1}{2}\int_{\Omega} \big(|\nabla_\G v_{n}|^{2}-\lambda_{1} v_{n}^{2}\big)\ d\xi 
			- \frac{1}{\2} \int_{\Omega} (v_{n}+t(v_{n})e_{1})^{\2}_{+}\ d\xi 
			- \int_{\Omega} f(v_{n}+t(v_{n})e_{1})\ d\xi
			&= m+o(1), \label{3.30} \\
			\int_{\Omega} \big(\nabla_\G v_{n} \cdot \nabla_\G \phi -\lambda_{1} v_{n} \ \phi \big)\ d\xi
			- \int_{\Omega} (v_{n}+t(v_{n})e_{1})^{\2-1}_{+} \phi\ d\xi
			- \int_{\Omega} f\ \phi\ d\xi
			&= o(1), \label{3.31}
		\end{align}
		for all $\phi \in \h$. Passing to the limit in \eqref{3.31} and using \eqref{3.26} together with Lemma \ref{lem3.4}, we find that $v_0$ satisfies
		\begin{equation}\nonumber
			-\Delta_\G v_0 = \lambda_1 v_0 + (v_0 + t(v_0)e_1)^{\2-1}_+ + f
		\end{equation}
		in the weak sense, along with
		\begin{equation}\nonumber
			\int_\Omega \big((v_0 + t(v_0)e_1)^{\2-1}_+ e_1 + f e_1\big)\, d\xi = 0.
		\end{equation}
		It remains to show $v_0 \neq 0$. 
		We first claim that
		\begin{equation}\label{3.35}
			\lim_{n\to\infty} t(v_{n}) = t(v_{0}).
		\end{equation}
		Indeed, if not, suppose $\lim_{n} t(v_{n}) = t_{1} \neq t(v_{0})$. Then by \eqref{3.6},
		\[
		\int_{\Omega} (v_{n}+t(v_{n})e_{1})^{\2-1}_{+} e_{1}\ d\xi
		= -\int_{\Omega} fe_{1}\ d\xi
		= \int_{\Omega} (v_{0}+t(v_{0})e_{1})^{\2-1}_{+} e_{1}\ d\xi,
		\]
		which would yield
		\[
		\int_{\Omega} (v_{0}+t_{1}e_{1})^{\2-1}_{+} e_{1}\ d\xi
		= \int_{\Omega} (v_{0}+t(v_{0})e_{1})^{\2-1}_{+} e_{1}\ d\xi,
		\]
		a contradiction. Thus \eqref{3.35} holds. Let $w_n \coloneq v_n - v_0$. Using the Br\'ezis-Lieb lemma in \eqref{3.30} and \eqref{3.31}, we obtain
		\begin{align}
			F(v_n) &= F(v_0) + \frac{1}{2} \int_\Omega |\nabla_\G w_n|^2\,d\xi - \frac{1}{\2} \int_\Omega |w_n|^{\2}\,d\xi = m+ o(1), \label{3.37} \\
			F'(v_n) v_n & = \int_\Omega |\nabla_\G w_n|^2\,d\xi - \int_\Omega |w_n|^{\2}\,d\xi = o(1). \label{3.38}
		\end{align}
		Let $b \coloneq \lim_{n\to\infty} \int_\Omega |\nabla_\G w_n|^2\,d\xi \ge 0$. If $b=0$, then $v_n \to v_0$ strongly and $v_0 \neq 0$ follows. Suppose $b>0$. By the Sobolev inequality and \eqref{3.38}, we have
		\begin{equation}\nonumber
			b = \lim_{n\to\infty} \int_\Omega |\nabla_\G w_n|^2\,d\xi \ge S_\G \lim_{n\to\infty} \left(\int_\Omega |w_n|^{\2}\,d\xi \right)^{2/\2} \ge S_\G b ^{2/\2}.
		\end{equation}
		That is \begin{equation}\label{3.39}
			b \geq S_\G^{Q/2}.
		\end{equation}
		Combining \eqref{3.25}, \eqref{3.37} and \eqref{3.39}, we obtain
		\[
		m = \lim_{n\to\infty} F(v_n) \ge F(v_0) + \frac{1}{Q} b \ge F(v_0) + \frac{1}{Q} S_\G^{Q/2} > F(v_0)+ m,
		\]
		It follows that $F(v_0) <0$. Since $F(0) >0$, this implies that $v_0 \neq 0$. Finally, Lemma~\ref{lem3.2} ensures that $F(v) \ge \alpha_0 > 0$ on the boundary $\|\nabla_\G v\|_{L^2(\Omega)} = \sigma_0$, so $v_0 \in B_d(0, \sigma_0)$. Hence $v_0$ is a nontrivial minimizer of $F$ and yields a weak solution of \eqref{3.1}.
	\end{proof}
	
	\section{Proof of Theorem \ref{thm1.3}}
	\label{sec4}
	In this section, we study the bifurcation of the set of solutions of \eqref{1.1}. Let $u_t(\lambda) = u_t$ denote the nonpositive solution obtained in Section \ref{sec2}. 
	If $f = t e_1 + h$ with $h \in \ker(-\Delta_\G - \lambda I)^\perp$, then $u_t(\lambda)$ is well defined for all $\lambda \neq \lambda_1$. In the case $\lambda = \lambda_k$, $k \neq 1$, the set of solutions of \eqref{1.1} bifurcating from $(\lambda_k, u_t(\lambda_k))$ is equivalent to the set of solutions of \eqref{2.4} bifurcating from $(\lambda_k,0)$.  
	Let
	\[
	\mathcal{S}^- = \operatorname{span}\{e_1, \dots, e_k\}, 
	\qquad 
	\mathcal{S}^+ = (\mathcal{S}^-)^\perp.
	\]
	
	\begin{prop}
		Every eigenvalue $\lambda_k$ of $-\Delta_\G$ gives rise to a bifurcation point $(\lambda_k,0)$ of \eqref{2.4}. Consequently, Theorem \ref{1.3} follows.
	\end{prop} 
	\begin{proof}
		The conclusion follows from an abstract bifurcation theorem due to B\"ohme \cite{Bohme} and Marino \cite{Marino}, see also in \cite[Theorem 11.4]{Rabinowitz}. We consider a function $g: \R \times \R \to \R$ is defined by
		\[
		g(\lambda,s) = \psi(s)(s+u_t(\lambda))^{\2-1}_+ + (1-\psi(s)),
		\]
		with $\psi \in C^\infty(\mathbb{R})$ satisfy
		\[
		\psi(s) = 1 \quad \text{for } |s|\leq 1, 
		\qquad 
		\psi(s) = 0 \quad \text{for } |s|\geq 2, 
		\qquad 
		0 \leq \psi(s) \leq 1 \ \ \forall s.
		\]
		Then $g \in C^1$, and for bounded $\lambda$, we have $g(\lambda,s) = o(|s|)$ as $s \to 0$.
		Set
		\[
		\Phi(v) = \frac{1}{2}\int_\Omega |\nabla_\G v|^2 \ d\xi - \int_\Omega G(\lambda,v)\ d\xi,
		\qquad \text{ where }
		G(\lambda,v) = \int_0^v g(\lambda,t)\ dt,
		\]
		and $v \in \h$.  
		It is standard to check that $\Phi \in C^2$.  
		A critical point $u$ of $\Phi$ restricted to the manifold
		\[
		\mathcal{N} \coloneq \left\{u \in \h : \int_\Omega |u|^2\ d\xi = r^2\right\}
		\]
		is a weak solution of
		\[
		-\Delta_\G u - g(\lambda,u) = \mu u
		\]
		for some Lagrange multiplier $\mu$. Define the bilinear form
		\[
		(Av,\phi) = \int_\Omega \nabla_\G v \cdot \nabla_\G \phi \ d\xi,
		\qquad 
		B(v)\phi = \int_\Omega g(\lambda,v)\phi \ d\xi,
		\]
		for $\phi \in \h$. Let $2 < q < \2$ and define $\Omega_0 \coloneq \{\xi \in \Omega : |v(\xi)| \geq 2\}$ for $v \in \h$. Since $|v|^q\ge 2^q$ on $\Omega_0$, it follows that
		\[
		\int_\Omega |v|^q\,d\xi \ge 2^q|\Omega_0|,
		\qquad\text{hence}\qquad
		|\Omega_0|\le 2^{-q}\|v\|_{L^q(\Omega)}^q .
		\]
		Hence,
		\[
		|B(v)\phi| 
		\leq \int_{\Omega \setminus \Omega_0} |v|^{\2-1}|\phi|\ d\xi + \int_{\Omega_0} |\phi|\ d\xi.
		\]
		On $\Omega\setminus\Omega_0$, H\"older’s inequality together with the Sobolev embedding
		$\h\hookrightarrow L^q(\Omega)$ yields
		\[
		\int_{\Omega\setminus\Omega_0} |v|^{2^*_Q-1}|\phi|\,d\xi
		\le \|v\|^{\2-1}_{L^{2(\2-1)}(\Omega)}\|\nabla_\G \phi\|_{L^2(\Omega)}.
		\]
		On $\Omega_0$, using H\"older’s inequality and the estimate on $|\Omega_0|$, we have
		\[
		\int_{\Omega_0} |\phi|\,d\xi
		\le |\Omega_0|^{1/2}\|\phi\|_{L^2(\Omega)}
		\le C\|v\|_{L^q(\Omega)}^{q/2}\|\nabla_\G \phi\|_{L^2(\Omega)}.
		\]
		Combining the above estimates and using again the Sobolev embedding,
		we conclude that
		\[
		|B(v)\phi| 
		\le C\left(\|\nabla_\G v\|_{L^2(\Omega)}^{\2-1}+\|\nabla_\G v\|_{L^2(\Omega)}^{q/2}\right) \|\nabla_\G \phi\|_{L^2(\Omega)}
		= o(\|\nabla_\G v\|_{L^2(\Omega)})
		\quad \text{as } \|\nabla_\G v\|_{L^2(\Omega)}\to 0.
		\]
		This shows $\|B(v)\| = o(\|\nabla_\G v\|_{L^2(\Omega)})$. Therefore, by in \cite[Theorem 11.4]{Rabinowitz}, each eigenvalue of $-\Delta_\G$ is a bifurcation point of
		\begin{equation}
			-\Delta_\G v - g(\lambda,v) = \lambda v. \label{4.1}
		\end{equation}
		Since $g(\lambda,v) = o(\|\nabla_\G v\|_{L^2(\Omega)})$ and $\lambda$ is bounded, it follows from \eqref{4.1} that
		\[
		\|\nabla_\G v\|_{L^2(\Omega)} \leq C \|v\|_{L^2(\Omega)} = Cr.
		\]
		Finally, regularity arguments imply that if $r$ is sufficiently small, then
		\[
		\|v\|_{L^\infty(\Omega)} < 1,
		\qquad 
		g(\lambda,v) = (v+u_t(\lambda))^{\2-1}_+.
		\]
		This completes the proof.
	\end{proof}
	Next, we show that the bifurcation branch bends locally to the left.
	\begin{prop}
		If $(\lambda, v(\lambda))$, with $v(\lambda)\neq 0$, is a solution of \eqref{2.4} such that 
		\[
		\lambda \to \lambda_k, \quad k \neq 1, \qquad v(\lambda) \to 0,
		\]
		then $\lambda < \lambda_k$. Consequently, if $h \in \ker(-\Delta_\G - \lambda_k I)^\perp$ and $(\lambda,u(\lambda))$, with $u(\lambda)\neq 0$, is a solution of \eqref{1.1} such that
		\[
		\lambda \to \lambda_k, \quad k \neq 1, \qquad u(\lambda) \to u_t(\lambda_k),
		\]
		then necessarily $\lambda < \lambda_k$.
	\end{prop}
	
	\begin{proof}
		Let $u = v + w$ be a solution of \eqref{2.4}, where $v \in \mathcal{S}^-$ and $w \in \mathcal{S}^+$.  
		Multiplying \eqref{2.4} by $w-v$ and integrating by parts, we obtain
		\begin{align*}
			\int_{\Omega} \big(|\nabla_\G w|^2 - |\nabla_\G v|^2\big)\ d\xi 
			& = \int_{\Omega} (\lambda u + (u+u_t(\lambda))^{\2-1}_+) (w-v) \ d\xi \nonumber\\
			&= \int_{\Omega} \Big[ \lambda(w^2-v^2) + (v+w+u_t(\lambda))^{\2-1}_+ (w-v)\Big] \ d\xi.
		\end{align*}
		This implies
		\begin{equation}\label{4.3}
			\Big(1 - \tfrac{\lambda}{\lambda_{k+1}}\Big)\int_{\Omega} |\nabla_\G w|^2 \ d\xi
			- \Big(1 - \frac{\lambda}{\lambda_k}\Big)\int_{\Omega} |\nabla_\G v|^2 \ d\xi
			\leq \int_{\Omega} (v+w+u_t(\lambda))^{\2-1}_+ (w-v)\ d\xi.
		\end{equation}
		By the convexity of the function $s \mapsto (s+u_t(\lambda))^{2^*-1}_+$ and since $u_t$ is non negative, we estimate
		\begin{align}\label{4.4}
			\int_{\Omega} (v+w+u_t(\lambda))^{\2-1}_+(w-v)\ d\xi
			&= \int_{\Omega} (v+w+u_t(\lambda))^{\2-1}_+(2w-u)\ d\xi \nonumber \\
			&\leq \int_{\Omega} (2w+u_t(\lambda))^{\2}_+ \ d\xi 
			- \int_{\Omega} (u+u_t(\lambda))^{\2}_+ \ d\xi \nonumber \\
			&\leq \int_{\Omega} (2w+u_t(\lambda))^{\2}_+ \ d\xi 
			\leq \int_{\Omega} |2w|^{\2}\ d\xi \nonumber \\
			&\leq C \|\nabla_\G w\|_{L^2(\Omega)}^{\2}.
		\end{align}
		Combining \eqref{4.3} and \eqref{4.4}, we obtain
		\begin{equation}\nonumber
			\Bigg(\Big(1-\frac{\lambda}{\lambda_{k+1}}\Big) - C\|\nabla_\G w\|_{L^2(\Omega)}^{\2-2}\Bigg)\|\nabla_\G w\|_{L^2(\Omega)}^2
			- \Big(1-\frac{\lambda}{\lambda_k}\Big)\|\nabla_\G v\|_{L^2(\Omega)}^2 \leq 0.
		\end{equation}
		Suppose, by contradiction, that $\lambda \geq \lambda_k$.  
		Since $\frac{\lambda}{\lambda_k}-1 \ge 0$ and $u=v+w\neq 0$, we must have $w\neq 0$.  
		Hence,
		\begin{equation}\label{4.5}
			\Big(1 - \frac{\lambda}{\lambda_{k+1}}\Big) \leq C\|\nabla_\G w\|_{L^2(\Omega)}^{\2-2} \leq C\|\nabla_\G u\|_{L^2(\Omega)}^{\2-2}.
		\end{equation}
		Letting $\lambda\to\lambda_k$ and $u\to0$ in $\h$, the left-hand side
		of \eqref{4.5} converges to
		$1-\frac{\lambda_k}{\lambda_{k+1}}>0$, while the right-hand side tends to zero, which is impossible. Therefore, we must have $\lambda < \lambda_k$.  
	\end{proof}
	
	\section*{Acknowlegement} 
	Suman Kanungo acknowledges the financial aid from Council of Scientific \& Industrial Research (CSIR), Government of India, File No. 09/1237(15789)/2022-EMR-I. The research of Pawan Kumar Mishra is supported by the Science and Engineering Research Board, Government of India, Grant No. MTR/2022/000495.
	

\begin{thebibliography}{99}
		
		\bibitem{Ambrosetti} A. Ambrosetti, G. Prodi, On the inversion of some differentiable mappings with singularities between Banach spaces, Ann. Mat. Pura Appl. 93 (1972) 231-246.
		
		\bibitem{Ambrosio} V. Ambrosio, T. Isernia, The critical fractional Ambrosetti-Prodi problem, Rend. Circ. Mat. Palermo, II. Ser 71 (2022) 1107-1132.
		
		\bibitem{Birindelli} I. Birindelli, I. Capuzzo Dolcetta, A. Cutr\'i, Liouville theorems for semilinear equations on the Heisenberg group, Ann. Inst. H. Poincar\'e C Anal. Non Lin\'eaire 14 (1997) 295-308.
		
		\bibitem{Birindelli1} I. Birindelli, J.V. Prajapat, Nonlinear Liouville theorems in the Heisenberg group via the moving plane method, Comm. Partial Differential Equations 24 (1999) 1875-1890. 
		
		\bibitem{Bisci} G.M. Bisci, D. Repovs, Yamabe-type equations on Carnot groups, Potential Anal. 46 (2017)
		369-383.
		
		\bibitem{Bohme} R. B\"ohme, Die L\"osung der Verzwergungsgleichungen f\"ur nichtlineare Eigenwertprobleme, Math. Z. 127 (1972) 105-126.
		
		\bibitem{Bonfiglioli} A. Bonfiglioli, E. Lanconelli, F. Uguzzoni, Stratified Lie groups and potential theory for their sub-Laplacians, Berlin, Heidelberg: Springer Berlin Heidelberg (2007).
		
		\bibitem{BonfiglioliUguzzoni} A. Bonfiglioli, F. Uguzzoni, Nonlinear Liouville theorems for some critical problems on H-type groups, J. Funct. Anal. 207 (2004) 161-215.
		
		\bibitem{Bony} J.M. Bony, Principe du maximum, in\'egalit\'e de Harnack et unicit\'e du probl\'eme de Cauchy pour les op\'erateurs elliptiques d\'eg\'en\'er\'es, Ann. Inst. Fourier Grenobles 19 (1969) 277-304.
		
		\bibitem{Brezis} H. Br\'ezis, L. Nirenberg, Positive solutions of nonlinear elliptic equations involving critical Sobolev exponents, Commun. Pure Appl. Math. 36 (1983) 437-477.
		
		\bibitem{Brandolini} L. Brandolini, M. Rigoli, A.G. Setti, Positive solutions of Yamabe-type equations on the Heisenberg group, Duke Math. J. 91 (1998) 241-296.
		
		\bibitem{Calanchi} M. Calanchi, B. Ruf, Elliptic equations with one-sided critical growth, Electron. J. Differential Equations (2002) 1-21.
		
		\bibitem{Citti} G. Citti, Semilinear Dirichlet problem involving critical exponent for the Kohn Laplacian,
		Ann. Mat. Pura. Appl. 169 (1995) 375-392.
		
		\bibitem{Citti1} G. Citti, F. Uguzzoni, Critical semilinear equations on the Heisenberg group: the
		effect of the topology of the domain, Nonlinear Anal. 46 (2001) 399-417. 
		
		\bibitem{Cuesta} M. Cuesta, D.G. de Figueiredo, P.N. Srikanth, On a resonant-superlinear elliptic problem, Calc. Var. Partial Differ. Equ. 17 (2003) 221-233.
		
		\bibitem{Deng} Y.B. Deng, On the superlinear Ambrosetti-Prodi problem involving critical Sobolev exponents, Nonlinear Anal. 17 (1991) 1111-1124.
		
		\bibitem{Figueiredo1} D.G. de Figueiredo, On the Superlinear Ambrosetti-Prodi Problem, Nonlinear Analysis TMA 8 (1984) 655-665. 
		
		\bibitem{Figueiredo} D.G. de Figueiredo, Y. Jianfu, Critical superlinear Ambrosetti-Prodi problems, Topol. Methods Nonlinear Anal. 14 (1999) 59-80.
		
		\bibitem{Folland} G.B. Folland, Subelliptic estimates and function spaces on nilpotent Lie groups, Ark. Mat. 13 (1975) 161-207.
		
		\bibitem{Garagnani} E. Garagnani, F. Uguzzoni, A multiplicity result for a degenerate-elliptic equation
		with critical growth on noncontractible domains, Topol. Meth. Nonlinear Anal. 22 (2003) 53-68.
		
		\bibitem{Garofalo1} N. Garofalo, E. Lanconelli, Existence and nonexistence results for semilinear equations on the Heisenberg group, Indiana Univ. Math. J. 41 (1992) 71-98.
		
		\bibitem{Garofalo} N. Garofalo, D. Vassilev, Regularity near the characteristic set in the non-linear Dirichlet problem and conformal geometry of sub-Laplacians on Carnot Groups, Math. Ann. 318 (2000) 453-516.
		
		\bibitem{Garofalo1} N. Garofalo, D. Vassilev, Symmetry properties of positive entire solutions of Yamabe type
		equations on groups of Heisenberg type, Duke Math. J. 106 (2001) 411-448.
		
		\bibitem{Hormander} L. H\"ormander, Hypoelliptic second order differential equations, Acta Math. 119 (1967) 147-171.
		
		\bibitem{Jerison1} D. Jerison, J. Lee, The Yamabe problem on CR manifolds, J. Differential Geom. 25, (1987) 167-197. 
		
		\bibitem{Jerison2} D. Jerison, J. Lee, Extremals for the Sobolev inequality on the Heisenberg group and the CR Yamabe problem, J. Amer. Math. Soc. 1 (1988) 1-13. 
		
		\bibitem{Jerison3} D. Jerison, J. Lee, Intrinsic CR normal coordinates and the CR Yamabe problem, J. Differ. Geom. 29 (1989) 303-343.
		
		\bibitem{Lanconelli} E. Lanconelli, F. Uguzzoni, Asymptotic behaviour and non-existence theorems for semilinear Dirichlet problems involving critical exponent on unbounded domains of the Heisenberg group, Boll. Unione Mat. Ital. 1 (1998) 139-168.
		
		\bibitem{Lanconelli1} E. Lanconelli, F. Uguzzoni, Non-existence results for semilinear Kohn-Laplace equations in unbounded domains, Commun. Partial Diff. Equ. 25 (2000) 1703-1739.
			
		\bibitem{Loiudice2} A. Loiudice, Semilinear subelliptic problems with critical growth on Carnot groups, Manuscr. Math. 124 (2007) 247-259.
		
		\bibitem{Lu} G. Lu, J. Wei, On positive entire solutions to the Yamabe-type problem on the Heisenberg and stratified groups, Electron. Res. Announc. Amer. Math. Soc. 3 (1997) 83-89.
		
		\bibitem{Maalaoui} A. Maalaoui, V. Martino, Multiplicity result for a nonhomogeneous Yamabe type equation involving the Kohn Laplacian, J. Math. Anal. Appl. 399 (2013) 333-339.
		
		\bibitem{Maalaoui1} A. Maalaoui, V. Martino, A. Pistoia, Concentrating solutions for a sub-critical sub-elliptic problem, Differential Integral Equations 26 (2013) 1263-1274.
		
		\bibitem{Marino} A. Marino, La biforcazione nel caso variazionale, Confer. Sem. Mat. Univ. Bari 132
		(1977).
		
		\bibitem{Mawhin} J. Mawhin, M. Willem, Critical Point Theory and Hamiltonian Systems, Springer-Verlag (1989).
		
		\bibitem{Miyagaki} O.H. Miyagaki, D. Motreanu, F.R. Pereira, Multiple solutions for a fractional elliptic problem with critical growth, J. Differ. Equ. 269 (2020) 5542-5572.
		
		\bibitem{Paiva} F.O. de Paiva, A.E. Presoto, Semilinear elliptic problems with asymmetric nonlinearities, J.  Math. Anal. App. 409 (2014) 254-262.
		
		\bibitem{Palatucci} G. Palatucci, M. Piccinini, L. Temperini, Struwe’s global compactness and energy approximation of the critical Sobolev embedding in the Heisenberg group, Adv. Calc. Var. 18 (2025) 731-754.
		
		\bibitem{Rabinowitz} P.H. Rabinowitz, Minimax methods in critical point theory with applications to differential equations, Expository lectures from the CBMS Regional Conference, Series in Mathematics, American Mathematical Society, Vol. 65 (1986).
		
		\bibitem{Ribeiro} B. Ribeiro, E. Gloss, H. Pereira,  $(p, q)$-Laplacian Equations with Critical Growth and Jumping Nonlinearities: B. Ribeiro et al., Bull. Braz. Math. Soc. (N.S.) 56 (2025), p.30.
		
		\bibitem{Rothschild} L.P. Rothschild, E.M. Stein, Hypoelliptic differential operators and nilpotent groups, Acta Math. 137 (1976) 247-320.
		
		\bibitem{Ruf} B. Ruf, P.N. Srikanth, Multiplicity results for superlinear elliptic problems with partial interference with the spectrum, J. Math. Anal. Appl. 118 (1986) 15-23. 
		
		\bibitem{Sharma} L. Sharma, T. Mukherjee, On critical Ambrosetti-Prodi type problems involving mixed operator, J. Elliptic Parabol. Equ. 10 (2024) 1187-1216.
		
		\bibitem{Uguzzoni} F. Uguzzoni, A non-existence theorem for a semilinear Dirichlet problem involving critical exponent on halfspaces on the Heisenberg group, NoDEA Nonlinear Diff. Equ. Appl. 6 (1999) 191-206.
		
		\bibitem{Uguzzoni1} F. Uguzzoni, A note on Yamabe-type equations on the Heisenberg group, Hiroshima Math. J. 30 (2000) 179-189.
		
		\bibitem{Wei} J. Wei, S. Yan. Lazer-McKenna conjecture: The critical case, J. Funct. Anal. 244 (2007) 639-667.
		
	\end{thebibliography}
\end{document}